\documentclass[11pt,preprint]{imsart}

\usepackage[mathscr]{eucal}
\usepackage{graphicx}
\usepackage{latexsym}
\usepackage{amssymb}
\usepackage{psfrag}
\usepackage{epsfig}
\usepackage{enumerate}
\usepackage{amsthm}
\usepackage{amsmath}
\usepackage{eufrak}
\DeclareMathAlphabet{\mathpzc}{OT1}{pzc}{m}{it}

\include{psbox}

\usepackage{amsfonts}
\usepackage{graphicx}
\usepackage{color}
\usepackage[usenames,dvipsnames]{xcolor}
\usepackage[textsize=small]{todonotes}
\usepackage{mathrsfs}
\usepackage{hyperref}

\usepackage[numbers]{natbib}

\begin{document}

\newtheorem{theorem}{\bf Theorem}[section]
\newtheorem{definition}{\bf Definition}[section]
\newtheorem{corollary}[theorem]{\bf Corollary}
\newtheorem{lemma}[theorem]{\bf Lemma}
\newtheorem{assumption}{Assumption}
\newtheorem{condition}{\bf Condition}[section]
\newtheorem{proposition}[theorem]{\bf Proposition}
\newtheorem{definitions}{\bf Definition}[section]
\newtheorem{problem}{\bf Problem}

\theoremstyle{remark}
\newtheorem{example}{\bf Example}[section]
\newtheorem{remark}{\bf Remark}[section]

\newcommand{ \Ss}{\mathbb{S}}
\newcommand{\beginsec}{
\setcounter{lemma}{0} \setcounter{theorem}{0}
\setcounter{corollary}{0} \setcounter{definition}{0}
\setcounter{example}{0} \setcounter{proposition}{0}
\setcounter{condition}{0} \setcounter{assumption}{0}
\setcounter{remark}{0} }

\numberwithin{equation}{section}
\numberwithin{theorem}{section}



\newcommand{\Erdos}{Erd\H{o}s-R\'enyi }
\newcommand{\Meleard}{M\'el\'eard }
\newcommand{\Frechet}{Fr\'echet }

\newcommand{\one} {{\boldsymbol{1}}}
\newcommand{\half}{{\frac{1}{2}}}
\newcommand{\quarter}{{\frac{1}{4}}}
\newcommand{\lfl}{\lfloor}
\newcommand{\rfl}{\rfloor}
\newcommand{\Rd}{{\Rmb^d}}
\newcommand{\intR}{\int_\Rmb}
\newcommand{\intRd}{\int_\Rd}
\newcommand{\eps}{\varepsilon}
\newcommand{\lan}{\langle}
\newcommand{\ran}{\rangle}
\newcommand{\ti}{\tilde}
\newcommand{\veps}{\varepsilon}
\newcommand{\skp}{\vspace{\baselineskip}}
\newcommand{\noi}{\noindent}

\newcommand{\Amb}{{\mathbb{A}}}
\newcommand{\Bmb}{{\mathbb{B}}}
\newcommand{\Cmb}{{\mathbb{C}}}
\newcommand{\Dmb}{{\mathbb{D}}}
\newcommand{\Emb}{{\mathbb{E}}}
\newcommand{\Fmb}{{\mathbb{F}}}
\newcommand{\Gmb}{{\mathbb{G}}}
\newcommand{\Hmb}{{\mathbb{H}}}
\newcommand{\Imb}{{\mathbb{I}}}
\newcommand{\Jmb}{{\mathbb{J}}}
\newcommand{\Kmb}{{\mathbb{K}}}
\newcommand{\Lmb}{{\mathbb{L}}}
\newcommand{\Mmb}{{\mathbb{M}}}
\newcommand{\Nmb}{{\mathbb{N}}}
\newcommand{\Omb}{{\mathbb{O}}}
\newcommand{\Pmb}{{\mathbb{P}}}
\newcommand{\Qmb}{{\mathbb{Q}}}
\newcommand{\Rmb}{{\mathbb{R}}}
\newcommand{\Smb}{{\mathbb{S}}}
\newcommand{\Tmb}{{\mathbb{T}}}
\newcommand{\Umb}{{\mathbb{U}}}
\newcommand{\Vmb}{{\mathbb{V}}}
\newcommand{\Wmb}{{\mathbb{W}}}
\newcommand{\Xmb}{{\mathbb{X}}}
\newcommand{\Ymb}{{\mathbb{Y}}}
\newcommand{\Zmb}{{\mathbb{Z}}}

\newcommand{\Amc}{{\mathcal{A}}}
\newcommand{\Bmc}{{\mathcal{B}}}
\newcommand{\Cmc}{{\mathcal{C}}}
\newcommand{\Dmc}{{\mathcal{D}}}
\newcommand{\Emc}{{\mathcal{E}}}
\newcommand{\Fmc}{{\mathcal{F}}}
\newcommand{\Gmc}{{\mathcal{G}}}
\newcommand{\Hmc}{{\mathcal{H}}}
\newcommand{\Imc}{{\mathcal{I}}}
\newcommand{\Jmc}{{\mathcal{J}}}
\newcommand{\Kmc}{{\mathcal{K}}}
\newcommand{\lmc}{{\mathcal{l}}}
\newcommand{\Lmc}{{\mathcal{L}}}
\newcommand{\Mmc}{{\mathcal{M}}}
\newcommand{\Nmc}{{\mathcal{N}}}
\newcommand{\Omc}{{\mathcal{O}}}
\newcommand{\Pmc}{{\mathcal{P}}}
\newcommand{\Qmc}{{\mathcal{Q}}}
\newcommand{\Rmc}{{\mathcal{R}}}
\newcommand{\Smc}{{\mathcal{S}}}
\newcommand{\Tmc}{{\mathcal{T}}}
\newcommand{\Umc}{{\mathcal{U}}}
\newcommand{\Vmc}{{\mathcal{V}}}
\newcommand{\Wmc}{{\mathcal{W}}}
\newcommand{\Xmc}{{\mathcal{X}}}
\newcommand{\Ymc}{{\mathcal{Y}}}
\newcommand{\Zmc}{{\mathcal{Z}}}

\newcommand{\Ebf}{{\mathbf{E}}}

\newcommand{\abd}{{\boldsymbol{a}}}
\newcommand{\Abd}{{\boldsymbol{A}}}
\newcommand{\bbd}{{\boldsymbol{b}}}
\newcommand{\Bbd}{{\boldsymbol{B}}}
\newcommand{\betabd}{{\boldsymbol{\beta}}}
\newcommand{\cbd}{{\boldsymbol{c}}}
\newcommand{\Cbd}{{\boldsymbol{C}}}
\newcommand{\dbd}{{\boldsymbol{d}}}
\newcommand{\Dbd}{{\boldsymbol{D}}}
\newcommand{\ebd}{{\boldsymbol{e}}}
\newcommand{\Ebd}{{\boldsymbol{E}}}
\newcommand{\etabd}{{\boldsymbol{\eta}}}
\newcommand{\fbd}{{\boldsymbol{f}}}
\newcommand{\Fbd}{{\boldsymbol{F}}}
\newcommand{\gbd}{{\boldsymbol{g}}}
\newcommand{\Gbd}{{\boldsymbol{G}}}
\newcommand{\hbd}{{\boldsymbol{h}}}
\newcommand{\Hbd}{{\boldsymbol{H}}}
\newcommand{\ibd}{{\boldsymbol{i}}}
\newcommand{\Ibd}{{\boldsymbol{I}}}
\newcommand{\jbd}{{\boldsymbol{j}}}
\newcommand{\Jbd}{{\boldsymbol{J}}}
\newcommand{\kbd}{{\boldsymbol{k}}}
\newcommand{\Kbd}{{\boldsymbol{K}}}
\newcommand{\lbd}{{\boldsymbol{l}}}
\newcommand{\Lbd}{{\boldsymbol{L}}}
\newcommand{\mbd}{{\boldsymbol{m}}}
\newcommand{\Mbd}{{\boldsymbol{M}}}
\newcommand{\mubd}{{\boldsymbol{\mu}}}
\newcommand{\nbd}{{\boldsymbol{n}}}
\newcommand{\Nbd}{{\boldsymbol{N}}}
\newcommand{\nubd}{{\boldsymbol{\nu}}}
\newcommand{\Nalpha}{{\boldsymbol{N_\alpha}}}
\newcommand{\Nbeta}{{\boldsymbol{N_\beta}}}
\newcommand{\Ngamma}{{\boldsymbol{N_\gamma}}}
\newcommand{\obd}{{\boldsymbol{o}}}
\newcommand{\Obd}{{\boldsymbol{O}}}
\newcommand{\pbd}{{\boldsymbol{p}}}
\newcommand{\Pbd}{{\boldsymbol{P}}}
\newcommand{\phibd}{{\boldsymbol{\phi}}}
\newcommand{\phihatbd}{{\boldsymbol{\hat{\phi}}}}
\newcommand{\psibd}{{\boldsymbol{\psi}}}
\newcommand{\psihatbd}{{\boldsymbol{\hat{\psi}}}}
\newcommand{\qbd}{{\boldsymbol{q}}}
\newcommand{\Qbd}{{\boldsymbol{Q}}}
\newcommand{\rbd}{{\boldsymbol{r}}}
\newcommand{\Rbd}{{\boldsymbol{R}}}
\newcommand{\rhobd}{{\boldsymbol{\rho}}}
\newcommand{\sbd}{{\boldsymbol{s}}}
\newcommand{\Sbd}{{\boldsymbol{S}}}
\newcommand{\sigmabd}{{\boldsymbol{\sigma}}}
\newcommand{\tbd}{{\boldsymbol{t}}}
\newcommand{\Tbd}{{\boldsymbol{T}}}
\newcommand{\taubd}{{\boldsymbol{\tau}}}
\newcommand{\ubd}{{\boldsymbol{u}}}
\newcommand{\Ubd}{{\boldsymbol{U}}}
\newcommand{\vbd}{{\boldsymbol{v}}}
\newcommand{\Vbd}{{\boldsymbol{V}}}
\newcommand{\varphibd}{{\boldsymbol{\varphi}}}
\newcommand{\wbd}{{\boldsymbol{w}}}
\newcommand{\Wbd}{{\boldsymbol{W}}}
\newcommand{\xbd}{{\boldsymbol{x}}}
\newcommand{\Xbd}{{\boldsymbol{X}}}
\newcommand{\ybd}{{\boldsymbol{y}}}
\newcommand{\Ybd}{{\boldsymbol{Y}}}
\newcommand{\zbd}{{\boldsymbol{z}}}
\newcommand{\Zbd}{{\boldsymbol{Z}}}

\newcommand{\abar}{{\bar{a}}}
\newcommand{\Abar}{{\bar{A}}}

\newcommand{\bbar}{{\bar{b}}}
\newcommand{\Bbar}{{\bar{B}}}
\newcommand{\cbar}{{\bar{c}}}
\newcommand{\Cbar}{{\bar{C}}}
\newcommand{\dbar}{{\bar{d}}}
\newcommand{\Dbar}{{\bar{D}}}
\newcommand{\ebar}{{\bar{e}}}
\newcommand{\Ebar}{{\bar{E}}}
\newcommand{\ebdbar}{{\bar{\ebd}}}
\newcommand{\Ebdbar}{{\bar{\Ebd}}}
\newcommand{\Ebfbar}{{\bar{\Ebf}}}
\newcommand{\Embbar}{{\bar{\Emb}}}
\newcommand{\fbar}{{\bar{f}}}
\newcommand{\Fbar}{{\bar{F}}}
\newcommand{\Fmcbar}{{\bar{\Fmc}}}
\newcommand{\gbar}{{\bar{g}}}
\newcommand{\Gbar}{{\bar{G}}}
\newcommand{\Gammabar}{{\bar{\Gamma}}}
\newcommand{\Gmcbar}{{\bar{\Gmc}}}
\newcommand{\Hbar}{{\bar{H}}}
\newcommand{\ibar}{{\bar{i}}}
\newcommand{\Ibar}{{\bar{I}}}
\newcommand{\jbar}{{\bar{j}}}
\newcommand{\Jbar}{{\bar{J}}}
\newcommand{\Jmcbar}{{\bar{\Jmc}}}
\newcommand{\kbar}{{\bar{k}}}
\newcommand{\Kbar}{{\bar{K}}}
\newcommand{\lbar}{{\bar{l}}}
\newcommand{\Lbar}{{\bar{L}}}
\newcommand{\mbar}{{\bar{m}}}
\newcommand{\Mbar}{{\bar{M}}}
\newcommand{\mubar}{{\bar{\mu}}}
\newcommand{\Mmbbar}{{\bar{\Mmb}}}
\newcommand{\nbar}{{\bar{n}}}
\newcommand{\Nbar}{{\bar{N}}}
\newcommand{\Nmcbar}{{\bar{\Nmc}}}
\newcommand{\nubar}{{\bar{\nu}}}
\newcommand{\obar}{{\bar{o}}}
\newcommand{\Obar}{{\bar{O}}}
\newcommand{\omegabar}{{\bar{\omega}}}
\newcommand{\Omegabar}{{\bar{\Omega}}}
\newcommand{\pbar}{{\bar{p}}}
\newcommand{\Pbar}{{\bar{P}}}
\newcommand{\Pbdbar}{{\bar{\Pbd}}}
\newcommand{\Phibar}{{\bar{\Phi}}}
\newcommand{\Pmbbar}{{\bar{\Pmb}}}
\newcommand{\Pmcbar}{{\bar{\Pmc}}}
\newcommand{\psibar}{{\bar{\psi}}}
\newcommand{\psibdbar}{{\bar{\psibd}}}
\newcommand{\qbar}{{\bar{q}}}
\newcommand{\Qbar}{{\bar{Q}}}
\newcommand{\rbar}{{\bar{r}}}
\newcommand{\Rbar}{{\bar{R}}}
\newcommand{\Rmcbar}{{\bar{\Rmc}}}
\newcommand{\sbar}{{\bar{s}}}
\newcommand{\Sbar}{{\bar{S}}}
\newcommand{\sigmabar}{{\bar{\sigma}}}
\newcommand{\Smcbar}{{\bar{\Smc}}}
\newcommand{\tbar}{{\bar{t}}}
\newcommand{\Tbar}{{\bar{T}}}
\newcommand{\taubar}{{\bar{\tau}}}
\newcommand{\Thetabar}{{\bar{\Theta}}}
\newcommand{\Tmcbar}{{\bar{\Tmc}}}
\newcommand{\ubar}{{\bar{u}}}
\newcommand{\Ubar}{{\bar{U}}}
\newcommand{\vbar}{{\bar{v}}}
\newcommand{\Vbar}{{\bar{V}}}
\newcommand{\wbar}{{\bar{w}}}
\newcommand{\Wbar}{{\bar{W}}}
\newcommand{\xbar}{{\bar{x}}}
\newcommand{\Xbar}{{\bar{X}}}
\newcommand{\Xbdbar}{{\bar{\Xbd}}}
\newcommand{\Xitbar}{{\bar{X}_t^i}}
\newcommand{\Xjtbar}{{\bar{X}_t^j}}
\newcommand{\Xktbar}{{\bar{X}_t^k}}
\newcommand{\Xisbar}{{\bar{X}_s^i}}
\newcommand{\ybar}{{\bar{y}}}
\newcommand{\Ybar}{{\bar{Y}}}
\newcommand{\zbar}{{\bar{z}}}
\newcommand{\Zbar}{{\bar{Z}}}

\newcommand{\Ahat}{{\hat{A}}}

\newcommand{\bhat}{{\hat{b}}}
\newcommand{\etahat}{{\hat{\eta}}}
\newcommand{\fhat}{{\hat{f}}}
\newcommand{\Fmchat}{{\hat{\Fmc}}}
\newcommand{\ghat}{{\hat{g}}}
\newcommand{\hhat}{{\hat{h}}}
\newcommand{\Jhat}{{\hat{J}}}
\newcommand{\Jmchat}{{\hat{\Jmc}}}
\newcommand{\lambdahat}{{\hat{\lambda}}}
\newcommand{\lhat}{{\hat{l}}}
\newcommand{\muhat}{{\hat{\mu}}}
\newcommand{\nuhat}{{\hat{\nu}}}
\newcommand{\phihat}{{\hat{\phi}}}
\newcommand{\psihat}{{\hat{\psi}}}
\newcommand{\rhohat}{{\hat{\rho}}}
\newcommand{\Smchat}{{\hat{\Smc}}}
\newcommand{\Tmchat}{{\hat{\Tmc}}}
\newcommand{\Yhat}{{\hat{Y}}}

\newcommand{\atil}{{\tilde{a}}}
\newcommand{\Atil}{{\tilde{A}}}
\newcommand{\Amctil}{{\tilde{\Amc}}}
\newcommand{\btil}{{\tilde{b}}}
\newcommand{\Btil}{{\tilde{B}}}
\newcommand{\bbdtil}{{\tilde{\bbd}}}
\newcommand{\betabdtil}{{\tilde{\betabd}}}
\newcommand{\ctil}{{\tilde{c}}}
\newcommand{\Ctil}{{\tilde{C}}}
\newcommand{\dtil}{{\tilde{d}}}
\newcommand{\Dtil}{{\tilde{D}}}
\newcommand{\etil}{{\tilde{e}}}
\newcommand{\Etil}{{\tilde{E}}}
\newcommand{\ebdtil}{{\tilde{\ebd}}}
\newcommand{\etatil}{{\tilde{\eta}}}
\newcommand{\Ebdtil}{{\tilde{\Ebd}}}
\newcommand{\Ebftil}{{\tilde{\Ebf}}}
\newcommand{\ftil}{{\tilde{f}}}
\newcommand{\Ftil}{{\tilde{F}}}
\newcommand{\Fmctil}{{\tilde{\Fmc}}}
\newcommand{\gtil}{{\tilde{g}}}
\newcommand{\Gtil}{{\tilde{G}}}
\newcommand{\gammatil}{{\tilde{\gamma}}}
\newcommand{\Gammatil}{{\tilde{\Gamma}}}
\newcommand{\Gmctil}{{\tilde{\Gmc}}}
\newcommand{\htil}{{\tilde{h}}}
\newcommand{\Htil}{{\tilde{H}}}
\newcommand{\itil}{{\tilde{i}}}
\newcommand{\Itil}{{\tilde{I}}}
\newcommand{\jtil}{{\tilde{j}}}
\newcommand{\Jtil}{{\tilde{J}}}
\newcommand{\Jmctil}{{\tilde{\Jmc}}}
\newcommand{\ktil}{{\tilde{k}}}
\newcommand{\Ktil}{{\tilde{K}}}
\newcommand{\ltil}{{\tilde{l}}}
\newcommand{\Ltil}{{\tilde{L}}}
\newcommand{\lambdatil}{{\tilde{\lambda}}}
\newcommand{\mtil}{{\tilde{m}}}
\newcommand{\Mtil}{{\tilde{M}}}
\newcommand{\mutil}{{\tilde{\mu}}}
\newcommand{\ntil}{{\tilde{n}}}
\newcommand{\Ntil}{{\tilde{N}}}
\newcommand{\nbdtil}{{\tilde{\nbd}}}
\newcommand{\nutil}{{\tilde{\nu}}}
\newcommand{\nubdtil}{{\tilde{\nubd}}}
\newcommand{\otil}{{\tilde{o}}}
\newcommand{\Otil}{{\tilde{O}}}
\newcommand{\omegatil}{{\tilde{\omega}}}
\newcommand{\Omegatil}{{\tilde{\Omega}}}
\newcommand{\ptil}{{\tilde{p}}}
\newcommand{\Ptil}{{\tilde{P}}}
\newcommand{\Pbdtil}{{\tilde{\Pbd}}}
\newcommand{\Pmbtil}{{\tilde{\Pmb}}}
\newcommand{\Pitil}{{\tilde{\Pi}}}
\newcommand{\psitil}{{\tilde{\psi}}}
\newcommand{\qtil}{{\tilde{q}}}
\newcommand{\Qtil}{{\tilde{Q}}}
\newcommand{\rtil}{{\tilde{r}}}
\newcommand{\Rtil}{{\tilde{R}}}
\newcommand{\rbdtil}{{\tilde{\rbd}}}
\newcommand{\rhotil}{{\tilde{\rho}}}
\newcommand{\rhobdtil}{{\tilde{\rhobd}}}
\newcommand{\Rmctil}{{\tilde{\Rmc}}}
\newcommand{\stil}{{\tilde{s}}}
\newcommand{\Stil}{{\tilde{S}}}
\newcommand{\sigmatil}{{\tilde{\sigma}}}
\newcommand{\sigmabdtil}{{\tilde{\sigmabd}}}
\newcommand{\ttil}{{\tilde{t}}}
\newcommand{\Ttil}{{\tilde{T}}}
\newcommand{\Tmctil}{{\tilde{\Tmc}}}
\newcommand{\thetatil}{{\tilde{\theta}}}
\newcommand{\Thetatil}{{\tilde{\Theta}}}
\newcommand{\util}{{\tilde{u}}}
\newcommand{\Util}{{\tilde{U}}}
\newcommand{\vtil}{{\tilde{v}}}
\newcommand{\Vtil}{{\tilde{V}}}
\newcommand{\Vmctil}{{\tilde{\Vmc}}}
\newcommand{\varphitil}{{\tilde{\varphi}}}
\newcommand{\varphibdtil}{{\tilde{\varphibd}}}
\newcommand{\wtil}{{\tilde{w}}}
\newcommand{\Wtil}{{\tilde{W}}}
\newcommand{\xtil}{{\tilde{x}}}
\newcommand{\Xtil}{{\tilde{X}}}
\newcommand{\xitil}{{\tilde{\xi}}}
\newcommand{\Xittil}{{\tilde{X}^{i}_t}}
\newcommand{\Xistil}{{\tilde{X}^{i}_s}}
\newcommand{\ytil}{{\tilde{y}}}
\newcommand{\Ytil}{{\tilde{Y}}}
\newcommand{\ztil}{{\tilde{z}}}
\newcommand{\Ztil}{{\tilde{Z}}}
\newcommand{\zetatil}{{\tilde{\zeta}}}
\newcommand{\Zittil}{{\tilde{Z}^{i,N}_t}}
\newcommand{\Zistil}{{\tilde{Z}^{i,N}_s}}

\newcommand{\zetainf}{{\underline{\zeta}}}
\newcommand{\zetasup}{{\|\zeta\|_\infty}}

\begin{frontmatter}
\title{Large Deviations for Brownian Particle Systems with Killing.\thanks{Research supported in part by the National Science Foundation (DMS-1305120), the Army Research Office (W911NF-14-1-0331) and Defense Advanced Research Projects Agency (W911NF-15-2-0122).}}
 \runtitle{LDP for Brownian Motions with Killing}

\begin{aug}
 \author{Amarjit Budhiraja \and Wai-Tong (Louis) Fan \and Ruoyu Wu}
\end{aug}

\skp

\today

\skp

\begin{abstract}
Particle approximations for certain nonlinear and nonlocal reaction-diffusion equations are studied using a system of Brownian motions with killing. 
The system is described by a collection of i.i.d.\ Brownian particles where each particle is killed independently at a rate determined by the empirical sub-probability measure of the states of the particles alive. 
A large deviation principle (LDP) for such sub-probability measure-valued processes is established. 
Along the way a convenient variational representation, which is of independent interest, for expectations of nonnegative functionals of Brownian motions together with an i.i.d.\ sequence of random variables  is established. 
Proof of the LDP relies on this variational representation and weak convergence arguments.

\noi {\bf AMS 2000 subject classifications:} Primary
60F10, 60K35; secondary 60B10, 93E20.

\noi {\bf Keywords:} Large Deviations, Weakly interacting particle systems, Brownian particles with killing, nonlinear reaction-diffusion equations, variational representations.
\end{abstract}

\end{frontmatter}


\section{Introduction}
\label{sec_kill:intro}

The goal of this work is to study a large deviation principle (LDP) for an interacting particle system associated with the nonlinear and nonlocal reaction-diffusion equation
\begin{equation}
	\label{eq_kill:PDE}
	\frac{\partial u(t,x)}{\partial t}=\frac{1}{2}\Delta u(t,x) - \lan\zeta,\,u\ran\,u(t,x),\; (t,x) \in (0,\infty)\times \Rd,
	\; \lim_{t\downarrow 0}u(t,\cdot) = \delta_0(\cdot),
\end{equation}
where $\Delta$ is the $d$-dimensional Laplacian operator, $\delta_0$ is the Dirac measure at $0$ and $\zeta \colon \Rd \to \Rmb_+$ is a continuous function with sub-quadratic growth, namely
\begin{equation}
\label{eq_kill:zeta}
	0 \le \zeta(x) \le C_\zeta (1 + \|x\|^p) \mbox{ for all} \, x \in \Rd
\end{equation}
for some $0 \le p < 2$ and $C_\zeta \in (0,\infty)$.
Here $\lan\zeta,\,u\ran$ denotes the integral $\int_{\Rd} \zeta(x) u(t,x) dx$.
Roughly speaking the system is described by $n$ independent Brownian particles where each particle is killed independently at a rate determined by the empirical measure of current particle states.  
More precisely, let $\{X_i\}_{i \ge 1}$ be a sequence of i.i.d.\ exponential random variables with rate $1$ and let $\{B_i(t), t \ge 0\}_{i \ge 1}$ be independent $d$-dimensional standard Brownian motions independent of $\{X_i\}_{i \ge 1}$. 
Define for $t \ge 0$ the random sub-probability measure $\mu^n(t)$ as the solution to the following equation
\begin{equation}
	\label{eq_kill:mu^n(t)_intro}
	\mu^n(t) = \frac{1}{n} \sum_{i=1}^n \delta_{B_i(t)} \one_{\left\{ X_i > \int_0^t \lan\zeta,\,\mu^n(s)\ran\,ds\right\}}.
\end{equation}
Since a.s., we can enumerate $\{X_i\}_{i=1}^n$ in a strictly increasing order, the unique solution of \eqref{eq_kill:mu^n(t)_intro} can be written explicitly in a recursive manner.
It can be checked (see Theorem \ref{thm_kill:LLN}) that $\mu^n \doteq \{\mu^n(t)\}_{t\in [0,T]}$ converges, in the Skorokhod path space $\Dmc$, in probability
to $\mu$ where for $t>0$, $\mu(t)$ has density $u(t, \cdot)$ given as the solution of \eqref{eq_kill:PDE}. 
Such particle systems are motivated by problems in biology, ecology, chemical kinetics, etc. For example, the simplest case 
where $\zeta \equiv 1$ corresponds to the case where the killing rate is proportional to the total number of particles alive and models a setting in which particles compete for a common resource. More general functions $\zeta$
are of interest as well and one interpretation of $\zeta(x)$ is the amount of resource consumed by  a particle in state $x$ (see for instance \cite{Freitas99} and Chapter 9 of \cite{Volpert14}).
Similar particle systems arise in problems of mathematical finance as models for self exciting correlated defaults \cite{CvMaZh}. 

In this work
we are interested in a Large Deviation Principle for the stochastic process $\mu^n$ in $\Dmc$.
A LDP, in addition to giving precise estimates on exponential decay rates of probabilities of deviations of the stochastic system from its law of large number limit, is  useful for developing accelerated Monte-Carlo methods for rare event simulation (see eg. \cite{DuWa1,DuSpZh}). Large deviation results for particle systems associated with nonlinear reaction-diffusion equations have been studied in many works\cite{Jona1993large,Bodineau2012large,Landim2015hydrostatics}.
However, these settings are very different from ours.
Specifically, in these works the models are described through jump-Markov processes on a discrete lattice with local interaction rules, the reaction operators in the PDE are local, and proof ideas are quite different. The setting considered in the current work is closer to that of particle systems with a weak interaction. 
Large deviation results for weakly interacting diffusions  go back to the classical work of Dawson and G{\"a}rtner \cite{DawsonGartner1987large} (see also Chapter 13 of \cite{FengKurtz06}) who studied uniformly nondegenerate diffusions that interact through a drift term that depends on the empirical measure of the particle states. In a recent work \cite{BudhirajaDupuisFischer2012}, using certain variational representations from \cite{BoueDupuis1998variational} for exponential functionals of Brownian motions,  a different proof of the LDP in \cite{DawsonGartner1987large} was provided. The work in \cite{BudhirajaDupuisFischer2012} in particular relaxed the assumption of uniform nondegeneracy and allowed mean field interaction in both drift and diffusion coefficients. Unlike \cite{DawsonGartner1987large}, proofs in
\cite{BudhirajaDupuisFischer2012} do not involve any space-time discretization or exponential probability estimates but rather rely on certain stochastic control representations and weak convergence arguments. A similar approach is taken in the current work although
the  interaction here is  of a   different form.  
The representation in \cite{BoueDupuis1998variational} is not well suited for our setting, since our particle system depends not only on the Brownian motions $\{B_i\}$ but also on the i.i.d. sequence  $\{X_i\}$. To address this, 
we begin by establishing a new variational representation (Proposition \ref{prop_kill:rep_general})
that allows the functional to depend, in addition to the Brownian motions, on an extra collection of i.i.d. random variables.
This representation plays a central role in the proof of the LDP and we believe that it is of independent interest and has potential applications for large deviation analysis of other interacting particle models (see Remark \ref{rem:rem921} for one such application).  
The starting point for the proof of the representation is the Bou\'{e}-Dupuis\cite{BoueDupuis1998variational} variational formula and the classical Donsker-Varadhan representation for exponential integrals. The key step in the proof is the interchange of a certain `integration' and `infimum' operation (see discussion above Proposition \ref{prop_kill:rep_general}) for which we need a careful measurable selection and construction of suitable controls (see Section \ref{sec:pfprop4.3}). Using the Laplace formulation of the LDP (see Chapter 1 of \cite{DupuisEllis2011weak}) the variational representation in Proposition \ref{prop_kill:rep_general} reduces the proof of the upper bound in the LDP to proving tightness and characterization of limit points of certain controlled empirical measures.
The proof of the lower bound relies on construction of asymptotically near optimal controlled stochastic processes with killing, together with a uniqueness argument for limit point that crucially uses the form of the killing function (see Lemma \ref{lem_kill:uniqueness}).

Although not explored here, we expect that similar techniques will be useful for settings where particle dynamics are given through general jump-diffusions and the killing function $\zeta$ is state dependent, i.e.
the integral on the right side of \eqref{eq_kill:mu^n(t)_intro} is replaced with $\int_0^t \lan\zeta_{B_i(s)},\,\mu^n(s)\ran\,ds$ where $\zeta(x,y)=\zeta_x(y)$ is a suitable $\Rmb_+$ valued function on
$\Rmb^d\times \Rmb^d$.  We note that $\mu^n$ can be written as a measurable function of the empirical measure of an i.i.d. collection. More precisely, denoting $\Gamma^n \doteq \frac{1}{n} \sum_{i=1}^n \delta_{(B_i,X_i)}$, we can represent $\mu^n = h(\Gamma^n)$ for a  measurable map from
$\Pmc(\Cmc([0,T]: \Rmb^d)\times [0,\infty))$ to $\Dmb([0,T]: \Mmc(\Rmb^d))$ (see notation in Section \ref{sec_kill:model}). Sanov's theorem gives a LDP for $\Gamma^n$, however the function $h$ is not continuous due to which one cannot apply the contraction principle to deduce a LDP for $\mu^n$.  Another approach that is common in the study of large deviations for weakly interacting particle systems is to consider a convenient absolutely continuous change of measure under which the large deviation analysis is more tractable and then deduce the LDP for the original collection by obtaining suitable estimates on the Radon-Nikodym derivative. In the current setting one way to implement such an approach 
would be to first appeal to  Sanov's theorem for a LDP for  $\Gammatil^n \doteq \frac{1}{n} \sum_{i=1}^n \delta_{(B_i,\sigmatil_i)}$, where $\sigmatil_i \doteq \inf \{ t : X_i \le \int_0^t \lan \zeta,\, \mu(s) \ran \,ds \}$ for $i = 1,\dotsc,n$ and $\mu$ as before is the law of large number limit of $\mu^n$, and
then establish a LDP for $\Gammabar^n \doteq \frac{1}{n} \sum_{i=1}^n \delta_{(B_i,\sigmabar_i^n)}$, where $\sigmabar_i^n \doteq \inf \{ t : X_i \le \int_0^t \lan \zeta,\, \mu^n(s) \ran \,ds \}$ for $i = 1,\dotsc,n$, by estimating the Radon-Nikodym derivative
$\frac{dQ^n}{dP^n}$ where $Q^n$ and $P^n$ are the probability laws of $\Gammatil^n$ and $\Gammabar^n$ respectively.
However  obtaining the required estimates on $\frac{dQ^n}{dP^n}$ for carrying out this approach does not appear to be straightforward.

Finally, we remark that both the LDP result and the law of large numbers result are with respect to an ``annealed measure" $\Pbd$ (defined in Section \ref{sec_kill:model}) under which the randomness of $\{X_i\}_{i \in \Nmb}$ is averaged out. Although not pursued here, it is also of interest to consider a ``quenched" version of the LDP, in analogy to results for random walks in random environments and for random polymer models; see for instance \cite{CGZ00} and \cite{EJ15} respectively. 

The paper is organized as follows.
Section~\ref{sec_kill:model} introduces the basic model and notation.
In this section we also present a law of large numbers  (Theorem \ref{thm_kill:LLN}) for $\mu^n$.
Section~\ref{sec_kill:LDP} gives the main large deviations result of this work, namely Theorem~\ref{thm_kill:LDP}.  
Rest of the work is devoted to the proof of Theorem~\ref{thm_kill:LDP}.
We begin in Section~\ref{sec_kill:var_rep_general} by presenting a variational representation for functionals of independent Brownian motions and an i.i.d.\ sequence of random variables. Proof of this representation is given in Section \ref{sec:pfprop4.3}.
Tightness of controls and related processes in this representation is argued in Section~\ref{sec_kill:tightness_control}.
Using the representation and the tightness results from Section~\ref{sec_kill:tightness_control}, Section~\ref{sec_kill:Lapuppbd} proves the Laplace principle upper bound for the LDP in Theorem~\ref{thm_kill:LDP} while Section~\ref{sec_kill:Laplowbd} gives the proof of the lower bound.  
Compactness of sub-level sets of the candidate rate function is established in Section~\ref{sec_kill:rate_function_pf}.
The large deviation principle for $\mu^n$ is immediate on combining results of Sections~\ref{sec_kill:Lapuppbd},~\ref{sec_kill:Laplowbd} and~\ref{sec_kill:rate_function_pf}.
Finally Appendix \ref{sec_kill:append} provides the proof of Theorem \ref{thm_kill:LLN}.

\section{Model and notation}
\label{sec_kill:model}

The following notation will be used.
For a Polish space $(\Smb,d(\cdot,\cdot))$, denote the corresponding Borel $\sigma$-field by $\Bmc(\Smb)$.
For a signed measure $\nu$ on $\Smb$ and $\nu$-integrable function $f \colon \Smb \to \Rd$, we will use $\langle \nu,f \rangle$ and $\langle f, \nu \rangle$ interchangeably for $\int f \, d\nu$.
Let $\Pmc(\Smb)$ [resp.\ $\Mmc(\Smb)$] be the space of probability measures [resp.\ sub-probability measures] on $\Smb$, equipped with the topology of weak convergence.
A convenient metric for this topology is the bounded-Lipschitz metric $d_{BL}$, which metrizes $\Pmc(\Smb)$ [resp.\ $\Mmc(\Smb)$] as a Polish space, defined as
\begin{equation*}
d_{BL}(\nu_1,\nu_2) \doteq \sup_{\|f\|_{BL} \le 1} | \langle \nu_1 - \nu_2, f \rangle |, \quad \nu_1, \nu_2 \in \Pmc(\Smb) \mbox{ [resp.\ $\Mmc(\Smb)$]}.
\end{equation*}
Here $\|\cdot\|_{BL}$ is the bounded Lipschitz norm, i.e.\ for $f \colon \Smb \to \Rmb$,
\begin{equation*}
	\label{eq:notation}
	\|f\|_{BL} \doteq \max \{\|f\|_\infty, \|f\|_L\}, \quad \|f\|_\infty \doteq \sup_{x \in \Smb} |f(x)|, \quad \|f\|_L \doteq \sup_{x \ne y} \frac{|f(x)-f(y)|}{d(x,y)}.
\end{equation*}
For $\nu_1, \nu_2 \in \Pmc(\Smb)$, we write $\nu_1 \ll \nu_2$ if $\nu_1$ is absolutely continuous with respect to $\nu_2$.
Denote by $\Mmb_b(\Smb)$ [resp.\ $\Cmb_b(\Smb)$] the space of real bounded $\Bmc(\Smb)/\Bmc(\Rmb)$-measurable functions [resp.\ real bounded and continuous functions].
Let $\Cmb_b^k(\Rmb^d)$ be the space of functions on $\Rmb^d$, which have continuous and bounded derivatives up to the $k$-th order.

Fix $T < \infty$. 
All stochastic processes will be considered over the time horizon $[0,T]$. 
We will use the notations $\{X_t\}$ and $\{X(t)\}$ interchangeably for stochastic processes.
For a Polish space $\Smb$, denote by $\Cmb([0,T]:\Smb)$ [resp.\ $\Dmb([0,T]:\Smb)$] the space of continuous functions [resp.\ right continuous functions with left limits] from $[0,T]$ to $\Smb$, endowed with the uniform [resp.\ Skorokhod] topology. 
Probability law of an $\Smb$-valued random variable $X$ will be denoted as $\Lmc(X)$. 
We will use $\Ebf^\Pbd$ for expected value under some probability law $\Pbd$ but when clear from the context $\Pbd$ will be suppressed from the notation.
We say a collection $\{ X_n \}$ of $\Smb$-valued random variables is tight if $\{ \Lmc(X_n) \}$ are tight in $\Pmc(\Smb)$.
Convergence of a sequence $\{X_n\}$ of $\Smb$-valued random variables in distribution to $X$ will be written as $X_n \Rightarrow X$. 

We will usually denote by $\kappa, \kappa_1, \kappa_2, \dotsc$, the constants that appear in various estimates within a proof. 
The value of these constants may change from one proof to another.
Let $\Cmc \doteq \Cmb([0,T]: \Rmb^d)$, $\Smc \doteq \Cmc \times \Rmb_+$ and $\Dmc \doteq \Dmb([0,T]: \Mmc(\Rd))$.
Denote by $\Nmb_0$ the set of nonnegative integers.

Let $\{B_i\}_{i \in \Nmb}$ be a sequence of i.i.d.\ $d$-dimensional standard Brownian motions, and $\{X_i\}_{i \in \Nmb}$ be a sequence of i.i.d.\ mean $1$ exponential random variables independent of $\{B_i\}_{i \in \Nmb}$, defined on $(\Omega, \Fmc, \Pbd)$.
Let $\Fmc_t$ be the $\Pbd$-completion of the $\sigma$-field generated by $\{B_i(s), X_i, s \le t, i \in \Nmb\}$.
Fix a continuous function $\zeta: \mathbb{R}^d \to \mathbb{R}_+$ satisfying \eqref{eq_kill:zeta}.
The main result of this work is a large deviation principle (LDP) for $\mu^n \doteq \{\mu^n(t)\}_{t \in [0,T]}$ in $\Dmc$, where $\mu^n(t)$ is defined in \eqref{eq_kill:mu^n(t)_intro}.
We begin with the following law of large numbers result.

Let $\theta \doteq \Lmc(X_1)$ be the exponential distribution (with rate $1$). Let $\mu_0 \doteq \Lmc(B_1)$ be the Wiener measure on $\Cmc$, $\mu_{0,t} \doteq \Lmc(B_1(t))$ be the marginal distribution of $\mu_0$ at time instant $t$, and $b(t) \doteq \langle \zeta , \, \mu_{0,t}\rangle$. 
Let $a(t)$ be the unique solution of the ordinary differential equation (ODE)
\begin{equation}
	\label{eq_kill:ODE}
	\dot{a}(t) = -a^2(t) b(t), \; a(0) = 1,
\end{equation}
namely $a(t) = \frac{1}{1+\int_0^t b(s)\,ds}$.
Note that from \eqref{eq_kill:zeta} it follows that
\begin{equation}
\label{eq_kill:zeta_2}
	\zeta(x) \le \Ctil_\zeta(1 + \|x\|^2), \; x \in \Rd
\end{equation} 
for some $\Ctil_\zeta \in (0,\infty)$.
In particular $b(t)$ and $a(t)$ are well-defined.
Define $\mu \colon [0, T] \to \Mmc(\Rmb^d)$ as $\mu(t) \doteq a(t)\mu_{0,t}$.
Clearly, $\mu \in \Cmb([0,T]: \Mmc(\Rmb^d))$ and for positive $t$, $\mu(t)$ is absolutely continuous with respect to Lebesgue measure.
Writing for $t \in (0,T]$, $\mu_t(dx)=u(t,x)dx$, it is easily checked that $u(t,x)$ is the solution to \eqref{eq_kill:PDE}.
Furthermore we have the following law of large numbers whose
proof is provided in Appendix~\ref{sec_kill:append}.

\begin{theorem} 
	\label{thm_kill:LLN}
	As $n \to \infty$, $\mu^n$ converges to $\mu$ in probability in $\Dmc$.
\end{theorem}
By a more careful analysis, for the proof of Theorem \ref{thm_kill:LLN}, the assumption that $\zeta$ has sub-quadratic growth can be weakened to allow for quadratic growth.  However in the proof of the large deviation result that we present in the next section, the condition on sub-quadratic growth of $\zeta$ is used in an important way (see proof of Lemma \ref{lem_kill:char_limit_2}) and thus we make this assumption throughout.



\section{Large deviation principle}
\label{sec_kill:LDP}

In this section we present our main large deviations result. We begin by introducing some canonical spaces and processes that will be convenient in the analysis.


\subsection{Canonical space and processes}
\label{sec_kill:canonical}

Let $\Rmc^W$ be the space of finite Borel measures $r$ on $\Rmb^d \times [0,T]$ such that $r(\Rmb^d\times [0,t]) = t$ for all $t \in [0,T]$ and
\begin{equation*}
	\int_{\Rmb^d\times[0,T]} \|y\| \, r(dy\,dt) < \infty.
\end{equation*}
Such a measure can be disintegrated as $r(dy\, dt) = r_t(dy) \,dt$, where $t \mapsto r_t(\cdot)$ is a measurable map from $[0,T]$ to $\Pmc(\Rmb^d)$.
We equip $\Rmc^W$ with the topology of weak convergence plus convergence of first moments.
This topology can be metrized with the Wasserstein distance of order $1$ and the space with this metric is Polish (cf.~\cite[Section 6.3]{Rachev1991probability}).

We now introduce the \textit{canonical space}
	$\Xi \doteq \Cmc \times \Cmc \times \Rmc^W \times \Rmb_+$
and \textit{canonical variables} on $(\Xi , \Bmc(\Xi))$:  
For $\xi = (\btil, b, \rho, \sigmatil) \in \Xi$,
\begin{equation}
	\bbdtil(\xi) \doteq \btil ,\; \bbd(\xi) \doteq b, \; \rhobd(\xi) \doteq \rho, \; \sigmabdtil(\xi) \doteq \sigmatil. \label{eq:eq752}
\end{equation}
Let $\{\Gmc_t\}_{0 \le t \le T}$ be the canonical filtration on $(\Xi, \Bmc(\Xi))$, namely
\begin{equation*}
	\Gmc_t \doteq \sigma \{ \bbdtil_s, \bbd_s, \rhobd(A \times [0,s]), s \le t, A \in \Bmc(\Rmb^d), \sigmabdtil\}, \; t \in [0,T].
\end{equation*}
For $\Theta \in \Pmc(\Xi)$, denote by $\Theta_{(i)}$ the marginal of $\Theta$ on the $i-$th coordinate, $i=1,2,3,4$.
For example $\Theta_{(1)}(A) \doteq \Theta(A \times \Cmc \times \Rmc^W \times \Rmb_+)$ for all $A \in \Bmc(\Cmc)$.

\subsection{Rate function and statement of the LDP}
\label{sec_kill:rate function LDP}

Let $\Pmc_\infty$ denote the collection of all $\Theta \in \Pmc(\Xi)$ such that $\Theta_{(4)} \ll \theta$ and under $\Theta$ the following hold:

	\noi (1) $\bbd$ is a $d$-dimensional standard $\Gmc_t$-Brownian motion.
	
	\noi (2) $\int_{\Rmb^d \times [0,T]} \|y\|^2 \,\rhobd (dy \, dt) < \infty$, a.s.
	
	\noi (3) $\bbdtil_t = \bbd_t + \int_{\Rmb^d \times [0,t]} y\, \rhobd (dy \, ds)$, for all $t \in [0,T]$, a.s.
	

\begin{remark}
	\label{rmk_kill:Pmc_infty}
	Note that the property $\Theta_{(4)} \ll \theta$ above implies that for $\Theta \in \Pmc_\infty$, $\Theta_{(4)}$ is absolutely continuous with respect to the Lebesgue measure on $\Rmb_+$.
	This observation will be used in the proof of Lemmas~\ref{lem_kill:uniqueness},~\ref{lem_kill:tightness_2} and~\ref{lem_kill:char_limit_2}.
\end{remark}

For $\Theta \in \Pmc(\Xi)$ let
\begin{equation*}
	\Jmc(\Theta) \doteq \Ebf^{\Theta} \left[ \frac{1}{2} \int_{\Rd \times [0,T]} \|y\|^2 \, \rhobd(dy \, dt) \right] + R(\Theta_{(4)} \| \theta),
\end{equation*}
where $R(\gamma\|\theta)$ for $\gamma \in \Pmc(\Rd)$ is the relative entropy of $\gamma$ with respect to $\theta$ defined as
\begin{align*}
	R(\gamma \| \theta) \doteq 
	\begin{cases}
		\int_{\Rd} \left( \log \frac{d\gamma}{d\theta} \right) d\gamma, & \gamma \ll \theta,  \\
	 	\hfil \infty, & \text{ otherwise.}	\\
	\end{cases}
\end{align*}
Note that with $\Theta_0$ the probability law of $(B_1, B_1, \delta_{0}(dy)\otimes dt, X_1)$, where $B_1$ is a standard $d$-dimensional Brownian motion and $X_1$ is a mean one exponential random variable independent of $B_1$, $\Jmc(\Theta_0) =0$.

For $\Theta \in \Pmc_\infty$ consider the equation, whose unknown is $\varpi_\Theta \colon [0,T] \to \Mmc(\Rd)$,
\begin{equation}
	\label{eq_kill:varpi_Theta}
	\lan f,\,\varpi_\Theta(t) \ran = \Ebf^\Theta \left[ f(\bbdtil_t) \one_{\left\{\sigmabdtil > \int_0^t \lan\zeta,\,\varpi_\Theta(s)\ran\,ds\right\}} \right], \mbox{ for all } f \in \Cmb_b(\Rd), t \in [0,T].
\end{equation}
Note that if a $\varpi_\Theta \colon [0,T] \to \Mmc(\Rd)$ satisfies the above equation for all continuous and bounded $f$
then it also satisfies the equality for all bounded measurable $f$ since
given a bounded function $f: \Rd \to \Rmb$ we can find  functions $f_n \in \Cmb_b(\Rd)$ such that $f_n(x) \to f(x)$ for a.e. 
$x \in \Rd$ and $\{f_n\}$ is uniformly bounded.
By monotone convergence theorem we then see that the above equality holds for all non-negative  measurable functions from
$\Rd$ to $\Rmb$ and finally the equality also holds for all measurable
  $f: \Rd \to \Rmb$ that satisfy 
$\Ebf^\Theta  |f(\bbdtil_t)|  <\infty$.

The following lemma shows that if $\Jmc(\Theta)<\infty$ there is at most one solution to \eqref{eq_kill:varpi_Theta}.

\begin{lemma}
	\label{lem_kill:uniqueness}
	Suppose $\Theta \in \Pmc_\infty$ satisfies $\Jmc(\Theta) < \infty$.
	Then there is at most one measurable map $\varpi_\Theta:[0,T] \to \Mmc(\Rd)$ that satisfies
	 \eqref{eq_kill:varpi_Theta}. 
	Moreover such a solution must necessarily be in $\Cmb([0,T]:\Mmc(\Rd))$.
\end{lemma} 

\begin{proof}
Let $\varpi_\Theta:[0,T] \to \Mmc(\Rd)$ be a measurable map that satisfies
 \eqref{eq_kill:varpi_Theta}.
	Fix $t \in [0,T]$. From \eqref{eq_kill:zeta_2} and properties of $\Pmc_\infty$,
	\begin{align}
		 \Ebf^\Theta \left[ \zeta(\bbdtil_t) \one_{\left\{\sigmabdtil > \int_0^t \lan\zeta,\,\varpi_\Theta(s)\ran\,ds\right\}} \right] 
		& \le \Ctil_\zeta \Ebf^\Theta (1+\|\bbdtil_t\|^2) \nonumber\\
		& \le 2\Ctil_\zeta \left[ 1 + \Ebf^\Theta \|\bbd_t\|^2 + \Ebf^\Theta \left\| \int_{\Rmb^d \times [0,t]} y\, \rhobd (dy \, ds) \right\|^2 \right] \nonumber\\
		& \le \kappa (1+\Jmc(\Theta)) <\infty.\label{eq:eq422}
	\end{align}
	Thus from the comments below \eqref{eq_kill:varpi_Theta}, $\lan \zeta,\,\varpi_\Theta(t) \ran$ is  finite and equals
	$$\Ebf^\Theta \left[ \zeta(\bbdtil_t) \one_{\left\{\sigmabdtil > \int_0^t \lan\zeta,\,\varpi_\Theta(s)\ran\,ds\right\}} \right]$$
for all $t\in [0,T]$.
	Furthermore, it follows from  Remark~\ref{rmk_kill:Pmc_infty} that for each $f \in \Cmb_b(\Rd)$ and $t,s \in [0,T]$, as $t \to s$,
	\begin{equation*}
		f(\bbdtil_t) \one_{\left\{\sigmabdtil > \int_0^t \lan\zeta,\,\varpi_\Theta(u)\ran\,du\right\}} \to f(\bbdtil_s) \one_{\left\{\sigmabdtil > \int_0^s \lan\zeta,\,\varpi_\Theta(u)\ran\,du\right\}}, \: \mbox{ a.s.\ } \Theta,
	\end{equation*}
	and hence $\lan f,\,\varpi_\Theta(t)\ran \to \lan f,\,\varpi_\Theta(s)\ran$.
	This implies that $\varpi_\Theta \in \Cmb([0,T]:\Mmc(\Rd))$.
	Now it remains to argue the uniqueness of $\varpi_\Theta$.
	Let $\pi_1, \pi_2 \in \Cmb([0,T]:\Mmc(\Rd))$ be two solutions of \eqref{eq_kill:varpi_Theta}.
	Consider functions $h(t) \doteq \lan\zeta,\,\pi_1(t)\ran - \lan\zeta,\,\pi_2(t)\ran$ and $H(t) \doteq \int_0^t h(s) \, ds$ for $t \in [0,T]$.
	We will argue via contradiction that $H(t) \equiv 0$.
	This will show $\pi_1 = \pi_2$, proving the desired uniqueness.
	Suppose $M \doteq \sup_{0 \le t \le T} |H(t)| > 0$.
	Without loss of generality we assume that $M = \sup_{0 \le t \le T} H(t)$.
	Since $H$ is a continuous function, there exists some $t \in [0,T]$ such that $H(t)=M$.
	Define $t^* \doteq \inf \{ t : H(t) = M \}$.
	Then $H(t^*) = M$.
	Since $H(0)=0$, we have $t^*>0$.
	Again by continuity of $H$, there exists $t_0 \in (0,t^*)$ such that for all $t \in [t_0,t^*]$, $H(t) > 0$, and hence
	\begin{equation*}
		\int_0^{t} \lan\zeta,\,\pi_1(s)\ran \,ds > \int_0^{t} \lan\zeta,\,\pi_2(s)\ran \,ds,\; \forall\, t \in [t_0,t^*].
	\end{equation*}
	Recalling that $\zeta \ge 0$, we have for all $t \in [t_0,t^*]$,
	\begin{align*}
		h(t) & = \lan\zeta,\,\pi_1(t)\ran - \lan\zeta,\,\pi_2(t)\ran \\
		& = \Ebf^\Theta \left[ \zeta(\bbdtil_t) \left( \one_{\left\{\sigmabdtil > \int_0^t \lan\zeta,\,\pi_1(s)\ran\,ds\right\}} - \one_{\left\{\sigmabdtil > \int_0^t \lan\zeta,\,\pi_2(s)\ran\,ds\right\}} \right) \right] \le 0.
	\end{align*}
	This implies that
	\begin{equation*}
		M = H(t^*) = H(t_0) + \int_{t_0}^{t^*} h(s)\,ds \le H(t_0),
	\end{equation*}
	which contradicts with the definition of $t^*$.
	Hence we must have $H(t) \equiv 0$ and this completes the proof.
\end{proof}

We now introduce the rate function associated with $\{\mu^n\}$. Let $\Pmc_\infty^1$ be the collection of all $\Theta \in \Pmc_\infty$ with $\Jmc(\Theta)<\infty$ and such that there is a solution $\varpi_\Theta$ of \eqref{eq_kill:varpi_Theta} (which is necessarily unique).
Define function $I \colon \Dmc \to [0, \infty]$ as
\begin{equation} 
	\label{eq_kill:rate_function}
	I(\pi) \doteq \inf_{\{\Theta \in \Pmc_{\infty}^1 \colon \varpi_\Theta = \pi\}}  \Jmc(\Theta), \quad \pi \in \Dmc .
\end{equation}
As is usual, infimum over an empty set is taken to be $\infty$.

The following is the main result of this work.
\begin{theorem}
	\label{thm_kill:LDP}
	The sequence $\{\mu^n\}_{n \in \Nmb}$ satisfies a LDP in $\Dmc$ with rate function $I$.
\end{theorem}

\begin{remark}
	\label{rem:rem921}
The representation in Proposition \ref{prop_kill:rep_general} that is key to the proof of Theorem \ref{thm_kill:LDP} can be
used to also cover a more general 	 case where  the Brownian motions $B_i$ are replaced with   $Z_i(\cdot) \doteq Y_i + B_i(\cdot)$ where $\{Y_i\}$ are i.i.d.  and independent of $\{B_i,X_i\}$, and the distribution $m$ of $Y_1$ satisfies suitable integrability conditions. In this case the law of large numbers limit will solve the PDE \eqref{eq_kill:PDE} with initial condition $m$ and in the rate function governing the LDP, the term $R(\cdot \| \theta)$ will be replaced with a term of the form $R(\cdot \| \theta \otimes m)$.
\end{remark}

It follows from the second statement of  Lemma \ref{lem_kill:uniqueness} that $I(\pi) = \infty$ for all $\pi \in \Dmc \backslash \Cmb([0,T]:\Mmc(\Rd))$. 
Hence the effective domain of the rate function is contained in $\Cmb([0,T]:\Mmc(\Rd))$. By contraction principle (Remark (c) of Theorem 4.2.1 of \cite{DemboZeitouni09}), one obtains from Theorem \ref{thm_kill:LDP} the LDP in $\Mmc(\Rd)$ for $\{\mu^n_t\}_{n \in \Nmb}$ for each $t\in [0,T]$.
\medskip

Recalling the equivalence between a LDP and a Laplace principle (cf. \cite[Chapter 1]{DupuisEllis2011weak}), in order to prove Theorem~\ref{thm_kill:LDP} it suffices to show that:\\
\noindent (1) $I$ defined in \eqref{eq_kill:rate_function} is a rate function.\\
\noindent (2) For every $F \in \Cmb_b(\Dmc)$,
\begin{equation} 
	\label{eq_kill:maincvg}
	\lim_{n\to \infty}-\frac{1}{n}\log \Ebf \left[\exp\left(-n F(\mu^n)\right)\right] = \inf_{\pi \in \Dmc}\{ F(\pi) + I(\pi) \}.	
\end{equation}
Proof of item (1) is given in Section~\ref{sec_kill:rate_function_pf} while the proof of item (2) is carried out in two steps.
First in Section~\ref{sec_kill:Lapuppbd} we will prove the \textit{Laplace upper bound}:
\begin{equation} 
	\label{eq_kill:Lapuppbd}
	\liminf_{n\to \infty}-\frac{1}{n}\log \Ebf \left[\exp\left(-n F(\mu^n)\right)\right]	\ge
	\inf_{\pi \in \Dmc}\{ F(\pi) + I(\pi) \}.	
\end{equation}
The proof of \eqref{eq_kill:maincvg} is then completed in Section~\ref{sec_kill:Laplowbd} by proving the complementary \textit{Laplace lower bound}:
\begin{equation} 
	\label{eq_kill:Laplowbd}
	\limsup_{n\to \infty}-\frac{1}{n}\log \Ebf \left[\exp\left(-n F(\mu^n)\right)\right]	\le
	\inf_{\pi \in \Dmc}\{ F(\pi) + I(\pi) \}.	
\end{equation}

\section{A general variational representation formula}
\label{sec_kill:var_rep_general}

In order to prove Theorem~\ref{thm_kill:LDP} we need to study the asymptotics of
\begin{equation}
	\label{eq_kill:exp_form}
	-\frac{1}{n} \log \Ebf \left[\exp \big(- n F(\mu^n)\big)\right],
\end{equation}
where $F \in \Cmb_b(\Dmc)$.
For this we will use certain variational representations.
Note that $F(\mu^n)$ can be written as $\Psi(\Bbd^n, \Xbd^n)$, where $\Bbd^n \doteq \{B_i\}_{i=1}^n$ is an $nd$-dimensional Brownian motion, $\Xbd^n \doteq \{X_i\}_{i=1}^n$ is an $\Rmb_+^n$-valued random variable and $\Psi$ is a suitable map.
When $\Psi$ is just a function of the Brownian motion $\Bbd^n$, a variational representation for quantities as in \eqref{eq_kill:exp_form} was obtained in~\cite{BoueDupuis1998variational} (see also~\cite{BudhirajaDupuis2000variational} where a more convenient form that allows for an arbitrary filtration is given).
In this section, we will present an extension of this result that gives a variational representation for positive functionals of both $\Bbd^n$ and $\Xbd^n$.

Let $(\Omega,\Fmc,\{\Fmc_t\},\Pbd)$ be a $\Pbd$-complete filtered probability space on which are given a $d$-dimensional standard $\Fmc_t$-Brownian motion $B$ and an $\Fmc_0$-measurable random variable $X$, which takes values in a Polish space $\Smb$ and has law $\rho$.
Main result of this section is Proposition \ref{prop_kill:rep_general} which
gives a  variational representation for $-\log \Ebf \left[\exp \left(-f(B,X)\right) \right]$, where $f \in \Mmb_b(\Cmc \times \Smb)$. The proof is given in Section \ref{sec:pfprop4.3}. 
We will apply this result  in Sections \ref{sec_kill:tightness_control} -- \ref{sec_kill:Laplowbd} for the setting where $d$ is replaced by $nd$, $\Smb$ by $\Rmb_+^n$, and $\rho$ by $\theta^{\otimes n}$.

Note that since $X$ is $\Fmc_0$-measurable, $B$ and $X$ are independent.
Consider the probability space $(\Omegatil,\Fmctil,\Pbdtil)$, where $\Omegatil = \Cmc$, $\Fmctil = \Bmc(\Cmc)$ and $\Pbdtil$ is the $d$-dimensional Wiener measure.
Under $\Pbdtil$ the canonical coordinate process $\Wtil \doteq \{\Wtil_t(\omegatil) \doteq \omegatil(t), 0 \le t \le T\}$ is a standard $d$-dimensional Brownian motion with respect to the filtration $\{\Fmctil_t^\Wtil\} \doteq \{\sigma\{\Wtil(s):s \le t\}\}$.
Let $\{\Fmctil_t\}$ be the augmented filtration, namely $\Fmctil_t \doteq \sigma\{\Fmctil_t^\Wtil \cup \ti\Nmc\}$ and $\ti\Nmc$ is the collection of all $\ti\Pbd$-null sets.
For $f \in \Mmb_b(\Cmc \times \Smb)$, define
\begin{equation}
	\label{eq_kill:ftil(xbd)_def}
	\ftil(x) \doteq - \log \Ebftil \left[ \exp \left( -f(\Wtil, x) \right) \right], \quad x \in \Smb.
\end{equation}
It follows from the independence between $B$ and $X$ that
\begin{equation*}
	-\log \Ebf \left[\exp \left(-f(B,X)\right) \right] = -\log \Ebf \left[ \exp \left(-\ftil(X)\right) \right].
\end{equation*}
Applying classical results of Donsker--Varadhan (cf.~\cite{DupuisEllis2011weak} Proposition 1.4.2) to RHS, we have the following representation formula from the above equality
\begin{equation}
	\label{eq_kill:rep_complex}
	-\log \Ebf \left[\exp \left(-f(B,X)\right) \right] = \inf_{\Pi \in \Pmc(\Smb)} \left[ R(\Pi\|\rho) + \int_\Smb \ftil(x) \Pi(dx) \right].
\end{equation}
Consider the collection of processes
\begin{equation*}
	\Amctil \doteq \{ \psitil : \mbox{the process } \psitil \mbox{ is } \Fmctil_t \mbox{-progressively measurable and } \Ebftil \int_0^T \|\psitil(s)\|^2 \,ds < \infty\}.
\end{equation*}
From Theorem 3.1 in~\cite{BoueDupuis1998variational} we now have the following variational formula
\begin{equation}
	\label{eq_kill:ftil(xbd)}
	\ftil(x) = \inf_{\psitil \in \Amctil} \Ebftil \left[ \half \int_0^T \|\psitil(s)\|^2 \, ds + f\left(\Wtil + \int_0^\cdot \psitil(s) \, ds, x \right) \right],
\end{equation}
which together with $\eqref{eq_kill:rep_complex}$ gives
\begin{align}
	& -\log \Ebf \left[\exp \left(-f(B,X)\right) \right] \notag \\
	& \quad = \inf_{\Pi \in \Pmc(\Smb)} \left\{ R(\Pi\|\rho) \right. \notag \\
	& \qquad \left. + \int_\Smb \inf_{\psitil \in \Amctil} \Ebftil \left[ \half \int_0^T \|\psitil(s)\|^2 \, ds + f\left(\Wtil + \int_0^\cdot \psitil(s) \, ds, x \right) \right] \Pi(dx) \right\}. \label{eq_kill:rep_start}
\end{align}
The above representation is inconvenient to use in our large deviation proofs due to the infimum under the integral on the right side. This is taken care of in the following proposition that roughly says that the integral and the infimum can be interchanged and furthermore the minimal filtration $\Fmctil_t$ can be replaced by a general filtration $\{\Fmchat_t\}$. The fact that one can allow for a general filtration will be crucial in the proof of the lower bound (see Section \ref{sec_kill:Laplowbd}).

Let $(\Omegabar,\Fmcbar,\Pbdbar)$ be a probability space on which are given a $d$-dimensional standard Brownian motion $\Wbar$ and an $\Smb$-valued random variable $\Xbar$, which is independent of $\Wbar$, with law $\Pi$.
Let $\{\Fmchat_t\}$ be {\it any filtration} satisfying the usual conditions, such that $\Wbar$ is still a standard $d$-dimensional Brownian motion with respect to $\{\Fmchat_t\}$ and $\Xbar$ is $\Fmchat_0$-measurable.
One example of such a filtration is $\{\Fmcbar_t \doteq \sigma\{\Fmcbar_t^{\Wbar,\Xbar} \cup \Nmcbar\}\}$, where $\Nmcbar$ is the collection of all $\Pbdbar$-null sets and $\Fmcbar_t^{\Wbar,\Xbar} \doteq \sigma\{\Xbar, \Wbar(s):s \le t\}$.
Let $\Upsilon_{\Pi} \doteq (\Omegabar,\Fmcbar,\{\Fmchat_t\},\Pbdbar)$ and consider the following collection of processes
\begin{align*}
	\Amc(\Upsilon_{\Pi}) \doteq \{ \psihat : \mbox{the process } \psihat \mbox{ is } & \Fmchat_t \mbox{-progressively measurable} \\ & \qquad \mbox{and } \Ebfbar \int_0^T \|\psihat(s)\|^2 \,ds < \infty\}.
\end{align*}

\begin{proposition}
	\label{prop_kill:rep_general}
	Let $f \in \Mmb_b(\Cmc \times \Smb)$. Then
	\begin{align} 
		& -\log \Ebf \left[\exp \left(-f(B,X)\right) \right] \notag \\
		& \quad = \inf_{\Pi, \Upsilon_{\Pi}}\inf_{\psihat \in \Amc(\Upsilon_{\Pi})} \left\{ R(\Pi\|\rho) + \Ebfbar \left[ \half \int_0^T \|\psihat(s)\|^2 \, ds + f\left(\Wbar + \int_0^\cdot \psihat(s) \, ds, \Xbar \right) \right] \right\}, \label{eq_kill:rep_general}
	\end{align}
	where the outer infimum is over all $\Pi \in \Pmc(\Smb)$ and all systems $\Upsilon_{\Pi}$. 
\end{proposition}

\section{Variational representation and tightness properties}
\label{sec_kill:tightness_control}

In this section we will apply Proposition~\ref{prop_kill:rep_general} to establish a variational representation for 
\begin{equation*}
	-\frac{1}{n} \log \Ebf \left[\exp \big(- n F(\mu^n)\big)\right],
\end{equation*}
where $F \in \Cmb_b(\Dmc)$.

Given $\Pi^n \in \Pmc(\Rmb_+^n)$ such that $R(\Pi^n \| \theta^{\otimes n}) < \infty$, let $\Upsilon_{\Pi^n} \doteq (\Omegabar, \Fmcbar, \{\Fmchat_t\}_{0\le t \le T}, \Pbdbar)$ be a filtered probability space on  which one is given an $\Rmb_+^n$-valued $\Fmchat_0$-measurable random variable $\Sbd^n \doteq (\Sbd_i^n)_{i=1}^n$ with law  $\Pi^n$, and an $nd$-dimensional standard $\Fmchat_t$-Brownian motion $\betabd^n \doteq (\betabd_i^n)_{i=1}^n$.

Denote by $\Ebfbar$ the expectation under $\Pbdbar$.
We can disintegrate $\Pi^n$ as
\begin{equation*}
	\Pi^n(dx_1,\dotsc,dx_n) = \Pi_1^n(dx_1) \Pi_2^n(x_1,dx_2) \dotsm \Pi_n^n(x_1,\dotsc,x_{n-1},dx_n) \doteq \prod_{i=1}^n \nu_i^n(\xbd^n,dx_i),
\end{equation*}
where $\xbd^n \doteq (x_i)_{i=1}^n \in \Rmb_+^n$.
Define random measures 
\begin{equation}
	\label{eq_kill:nubd^n}
	\nubd_i^n \doteq \nu_i^n(\Sbd^n,\cdot), \; i=1,\dotsc,n. 
\end{equation}
Let
\begin{multline*}
	\Amc_n(\Upsilon_{\Pi^n}) \doteq \{ \psibd^n \doteq (\psibd_i^n)_{i=1}^n : \mbox{the process } \psibd_i^n \mbox{ is } \Rd \mbox{-valued } \\ \Fmchat_t \mbox{-progressively measurable and } \Ebfbar \sum_{i=1}^n \int_0^T \|\psibd_i^n(s)\|^2 \,ds < \infty\}.
\end{multline*}
For $\psibd^n \in \Amc_n(\Upsilon_{\Pi^n})$, consider the following controlled processes
\begin{align}
	\betabdtil^n & \doteq (\betabdtil_i^n)_{i=1}^n, \quad \betabdtil_i^n(t) \doteq \betabd_i^n(t) + \int_0^t \psibd_i^n(s) \, ds, \; i=1,\dotsc,n, \label{eq_kill:betabdtil^n} \\
	\mutil^n(t) & = \frac{1}{n} \sum_{i=1}^n \delta_{\betabdtil_i^n(t)} \one_{\left\{\Sbd_i^n > \int_0^t \lan\zeta,\,\mutil^n(s)\ran\, ds\right\}}. \label{eq_kill:mutil^n}
\end{align}
Note that since $R(\Pi^n \| \theta^{\otimes n}) < \infty$, $\Pi^n$ is absolutely continuous with respect to the Lebesgue
measure on $\mathbb{R}_+^n$.
Hence a.s.\ $\Pbdbar$ we can enumerate $\{\Sbd_i^n\}_{i=1}^n$ in a strictly increasing order, and the unique solution of \eqref{eq_kill:mutil^n} can be written explicitly in a recursive manner.
From Proposition \ref{prop_kill:rep_general} we have the following variational formula which will be the starting point of our proof to Theorem~\ref{thm_kill:LDP}.

\begin{lemma} \label{lem_kill:rep}
	Let $F \in \Cmb_b(\Dmc)$. Then
	\begin{align} 
		& -\frac{1}{n} \log \Ebf \left[\exp \left(-n F(\mu^n)\right)\right] \notag \\
		& \qquad = \inf_{\Pi^n, \Upsilon_{\Pi^n}} \inf_{\psibd^n \in \Amc_n(\Upsilon_{\Pi^n})} \Ebfbar \left[ \frac{1}{n} \sum_{i=1}^n \left( R(\nubd_i^n \| \theta) + \frac{1}{2} \int_0^T \|\psibd_i^n(s)\|^2 \, ds \right) + F(\mutil^n) \right], \label{eq_kill:rep_key}
	\end{align}
	where the outside infimum is over all $\Pi^n \in \Pmc(\mathbb{R}_+^n)$ and all systems $\Upsilon_{\Pi^n}$.
\end{lemma}

\begin{proof}
	Note that
	\begin{equation*}
		F(\mu^n) = F\left(\left\{\frac{1}{n} \sum_{i=1}^n \delta_{B_i(t)} \one_{\left\{ X_i > \int_0^t \lan\zeta,\,\mu^n(s)\ran\,ds\right\}} \right\}_{t\in[0,T]}\right) \doteq \Psi(\Bbd^n, \Xbd^n)
	\end{equation*} 
	for some $\Psi \in \Mmb_b(\Cmc^{n} \times \Rmb_+^n)$, where $\Bbd^n \doteq (B_i)_{i=1}^n$ and $\Xbd^n \doteq (X_i)_{i=1}^n$.
	The function $\Psi$ depends also on $n$ but that dependence is suppressed in the notation.
	With the same measurable function $\Psi$, $F(\mutil^n) = \Psi(\betabdtil^n, \Sbd^n)$ a.s.
	Applying Proposition~\ref{prop_kill:rep_general} with $d$, $\Smb$, $f$, $\Wbar$, $\Xbar$, $\Cmc$ and $\rho$  replaced by $nd$, $\Rmb_+^n$, $n\Psi$, $\betabd^n$, $\Sbd^n$, $\Cmc^n$ and $\theta^{\otimes n}$  respectively, we have
	\begin{align*} 
		& -\frac{1}{n} \log \Ebf \left[\exp \left(-n F(\mu^n)\right)\right] \\
		& \qquad = \inf_{\Pi^n, \Upsilon_{\Pi^n}} \inf_{\psibd^n \in \Amc_n(\Upsilon_{\Pi^n})} \left\{ \frac{1}{n} R(\Pi^n \| \theta^{\otimes n}) + \Ebfbar \left[ \frac{1}{2n} \sum_{i=1}^n \int_0^T \|\psibd_i^n(s)\|^2 \, ds + F(\mutil^n) \right] \right\}.
	\end{align*}
	Using chain rule for relative entropies (cf.~\cite{DupuisEllis2011weak}, Theorem C.3.1) it follows that
	\begin{equation*}
		R(\Pi^n \| \theta^{\otimes n}) = \Ebfbar \left[ \sum_{i=1}^n R(\nubd_i^n \| \theta) \right],
	\end{equation*}
	where $(\nubd_i^n)_{i=1}^n$ is defined in \eqref{eq_kill:nubd^n}.
	The result follows on combining the above two displays.
\end{proof}

We now present some tightness results that will play a key role in our proofs. For each $n \in \Nmb$, let $\Pi^n$,
$\Upsilon_{\Pi^n}$ (with the associated  $\Sbd^n$) and $\psibd^n \in \Amc_n(\Upsilon_{\Pi^n})$ be such that
\begin{equation} 
	\label{eq_kill:tightness_assumption}
	\sup_{n \in \Nmb} \Ebfbar \left[ \frac{1}{n} \sum_{i=1}^n \left( R(\nubd_i^n \| \theta) + \frac{1}{2} \int_0^T \|\psibd_i^n(s)\|^2 \, ds \right) \right] \doteq C_0 < \infty.
\end{equation}
For $i=1, \dotsc , n$, define $\Rmc^W$-valued random variables $\rhobd_i^n$ as
\begin{equation*}
	\rhobd_i^n(A \times [0,t]) \doteq \int_0^t \delta_{A}(\psibd_i^n(s)) \,ds, \; A \in \Bmc(\Rmb^d),\; t \in [0,T].
\end{equation*}
Recall the space $\Xi$ from Section \ref{sec_kill:canonical}. Define a sequence $\{(Q^n,\nu^n)\}_{n \in \Nmb}$ of $\Pmc(\Xi) \times \Pmc(\Rmb_+)$-valued random variables as
\begin{equation} 
	Q^n(A) \doteq \frac{1}{n} \sum_{i=1}^n \delta_{(\betabdtil_i^n, \betabd_i^n, \rhobd_i^n, \Sbd_i^n)}(A),\; A \in \Bmc(\Xi) \label{eq_kill:Q^n_general}
\end{equation}
and
\begin{equation}
	\nu^n(A) \doteq \frac{1}{n} \sum_{i=1}^n \nubd_i^n(A),\; A \in \Bmc(\Rmb_+). \label{eq_kill:nu^n_general}
\end{equation}
We then have the following tightness result.
In proving this result we will use some basic facts about tightness functions. Recall (see  \cite[Section 5.3]{DupuisEllis2011weak})  that a measurable map $\psi$ from a Polish space $\mathbb{S}$ to $\mathbb{R}$ is called a tightness function if it is bounded from below and has pre-compact level sets (i.e. for some $c \in \mathbb{R}$, $\psi(x) \ge c$ for all $x\in \mathbb{S}$ and the closure of $\{x\in \mathbb{S}: \psi(x)\le m\}$ is compact for all $m \in \mathbb{R}$). The following two facts about tightness functions will be used:
\begin{itemize}
	\item If $\psi$ is a tightness function on $\mathbb{S}$ and $\{\gamma_n\}\subset\Pmc(\mathbb{S})$ is such that
	$\sup_n \int_{\Smb} \psi(x) \gamma_n(dx) < \infty$, then $\{\gamma_n\}$ is tight.
	\item If $\psi$ is a tightness function on $\mathbb{S}$, then $\Psi: \Pmc(\Smb) \to \mathbb{R} \cup \{\infty\}$ defined as
	$\Psi(\gamma)\doteq \int_{\Smb} \psi(x) \gamma(dx)$ is a tightness function on $\Pmc(\Smb)$.
\end{itemize}
We will also appeal to the facts that a sequence $\{\Gamma^n\}$ of $\Pmc(\Smb)$-valued random variables is tight if and only if
$\gamma^n \doteq \Ebf\Gamma^n$ is a tight sequence in $\Pmc(\Smb)$ (see \cite[Chapter 1]{Sznitman1991}); and for Polish spaces $\Smb_i$, $i=1,2$, a sequence of probability measures on $\Smb_1\times \Smb_2$ is tight if the corresponding sequences of marginal distributions on  $\Smb_i$, $i=1,2$, are tight.
\begin{lemma} 
	\label{lem_kill:tightness_1}
	Suppose \eqref{eq_kill:tightness_assumption} holds.
	The sequence of random variables $\{(Q^n,\nu^n)\}_{n \ge 1}$ is tight in $\Pmc(\Xi) \times \Pmc(\Rmb_+)$.
\end{lemma}

\begin{proof}
	We first argue tightness of $\{Q^n\}$. It suffices to show that the sequences  $\{Q^n_{(i)}\}$ are tight for $i= 1, 2, 3, 4$.
	Since $\{\betabd_i^n\}$ are standard Brownian motions, $\{Q^n_{(2)}\}$ is clearly tight.  
Next we consider $Q^n_{(3)}$.
	Note that
	\begin{equation*}
		g_1(r) \doteq \int_{\Rmb^d\times [0,T]} \|y\|^2 \, r(dy\,dt),\; r \in \Rmc^W,
	\end{equation*}
	is a tightness function on $\Rmc^W$.
	This says that $G_1 \colon \Pmc(\Rmc^W) \to [0,\infty]$ defined as
	\begin{equation*}
		G_1(m) \doteq \int_{\Rmc^W} g_1(r) \, m(dr), \; m \in \Pmc(\Rmc^W)
	\end{equation*}
	is a tightness function on $\Pmc(\Rmc^W)$.  
	Next note that
	\begin{align*}
		\Ebfbar G_1(Q^n_{(3)}) &= \frac{1}{n} \sum_{i=1}^n \Ebfbar g_1(\rhobd_i^n) = \frac{1}{n} \sum_{i=1}^n \Ebfbar \int_{\Rmb^d\times [0,T]} \|y\|^2 \, \rhobd_i^n(dy\,dt) \\
		& = \frac{1}{n} \sum_{i=1}^n \Ebfbar \int_0^T \|\psibd_i^n(s)\|^2 \, ds \le 2C_0.
	\end{align*}
	This proves that $\{Q^n_{(3)}\}$ is tight.
	We now argue tightness of $\{Q^n_{(1)}\}$.
	Define a sequence $\{\Qtil^n\}_{n \in \Nmb}$ of $\Pmc(\Cmc)$-valued random variables as
	\begin{equation*}
		\Qtil^n(A) \doteq \frac{1}{n} \sum_{i=1}^n \delta_{\ubd_i^n}(A), \; A \in \Bmc(\Cmc),
	\end{equation*} 
	where
	\begin{equation*}
		\ubd^n_i(t) \doteq \int_{\Rd \times [0,t]} y \, \rhobd_i^n(dy\,ds) =  \int_0^t \psibd^n_i(s) \, ds, \; t \in [0,T].
	\end{equation*}
	We claim that $\{\Qtil^n\}_{n \in \Nmb}$ is a tight sequence.  
	To see this, note that $g_2 \colon \Cmc \to [0,\infty]$ defined as
	\begin{align*}
		g_2(x) \doteq \left \{
		\begin{array}{cl}
			\int_0^T \|\dot{x}(s)\|^2 \, ds + \|x(0)\|^2, & \mbox{ if $x$ is absolutely continuous,}  \\
		 	\infty, & \text{ otherwise,}	\\
		\end{array} \right.
	\end{align*}
	is a tightness function on $\Cmc$, from which it follows that $G_2 \colon \Pmc(\Cmc)\to [0,\infty]$, defined as
	\begin{equation*}
		G_2(m) \doteq \int_{\Cmc} g_2(x) \, m(dx), \; m \in \Pmc(\Cmc),
	\end{equation*}
	is a tightness function on $\Pmc(\Cmc)$.  
	Also,
	\begin{equation*}
		\Ebfbar(G_2(\Qtil^n)) = \frac{1}{n} \sum_{i=1}^n \Ebfbar g_2(\ubd^n_i) = \frac{1}{n} \sum_{i=1}^n \Ebfbar \int_0^T \|\psibd^n_i(s)\|^2 \, ds \le 2C_0.
	\end{equation*}
	This proves tightness of $\{\Qtil^n\}_{n \in \Nmb}$.
	Define the sequence $\{\Qbar^n_{(2,3)}\}_{n \in \Nmb}$ of $\Pmc(\Cmc \times \Cmc)$-valued random variables as
	\begin{equation*}
		\Qbar^n_{(2,3)}(A) \doteq \frac{1}{n} \sum_{i=1}^n \delta_{(\betabd_i^n, \ubd_i^n)}(A),\; A \in \Bmc(\Cmc \times \Cmc).
	\end{equation*}
	Then, from tightness of $\{Q^n_{(2)}\}$ and $\{\Qtil^n\}$ it follows that $\{\Qbar^n_{(2,3)}\}$ is tight.
	Next, noting that the map $g_3 \colon \Cmc \times \Cmc \to \Cmc$ defined as $g_3(x,u) \doteq x+u$ is continuous and that $Q^n_{(1)} = \Qbar^n_{(2,3)}\circ g_3^{-1}$, we get tightness of $\{Q^n_{(1)}\}$.
	
	Finally, tightness of $\{Q^n_{(4)}\}$ and $\{\nu^n\}$ can be proved using a standard argument as in the proof of Sanov's theorem (see e.g.~\cite[Theorem 2.2.1]{DupuisEllis2011weak}).
	However we provide details for sake of completeness.
	Note that $G_3(\cdot) \doteq R(\cdot \| \theta)$ is a convex tightness function on $\Pmc(\Rmb_+)$ (cf.~\cite[Lemma 1.4.3]{DupuisEllis2011weak}), and that by \eqref{eq_kill:tightness_assumption} and Jensen's inequality,
	\begin{equation*}
		\Ebfbar G_3(\nu^n) = \Ebfbar R(\nu^n \| \theta) \le \Ebfbar \left[ \frac{1}{n} \sum_{i=1}^n R(\nubd_i^n \| \theta) \right] \le C_0.
	\end{equation*}
	So $\{\nu^n\}$ is tight, which implies that $\{\Ebfbar \nu^n\}$ is also tight.
	Next, since $\nubd_i^n$ is the conditional distribution used to select $\Sbd_i^n$, for any $f \in \Mmb_b(\Rd)$,
	\begin{equation*}
		\Ebfbar \int_\Rd f(x) \, Q^n_{(4)}(dx) = \frac{1}{n} \sum_{i=1}^n \Ebfbar f(\Sbd_i^n) = \frac{1}{n} \sum_{i=1}^n \Ebfbar \int_\Rd f(x) \, \nubd_i^n(dx) = \Ebfbar \int_\Rd f(x) \, \nu^n(dx).
	\end{equation*}
	So $\Ebfbar Q^n_{(4)} = \Ebfbar \nu^n$, from which we have the desired tightness of $\{Q^n_{(4)}\}$.
\end{proof}

The following lemma gives a useful characterization of weak limit points of $\{(Q^n, \nu^n): n \in \Nmb\}$. 
Recall the space $\Pmc_{\infty}$ introduced in Section \ref{sec_kill:rate function LDP}.
\begin{lemma}
	\label{lem_kill:char_limit_1}
	Suppose \eqref{eq_kill:tightness_assumption} holds and suppose $(Q^n,\nu^n)$ converges along a subsequence, in distribution, to $(Q^*,\nu^*)$ given on some probability space $(\Omega^*, \Fmc^*, \Pbd^*)$. 
	Then $Q^*_{(4)} = \nu^*$, $Q^* \in \Pmc_{\infty}$, and $\Jmc(Q^*) < \infty$, a.s.\ $\Pbd^*$.
\end{lemma}

\begin{proof}
	The first statement, namely $Q^*_{(4)} = \nu^*$ a.s.\ $\Pbd^*$, can be proved using a standard martingale argument as in the proof of Sanov's theorem (see e.g.~\cite[Theorem 2.2.1]{DupuisEllis2011weak}). We omit the details.
	
	In order to check the second statement that $Q^* \in \Pmc_{\infty}$ a.s.\ $\Pbd^*$, we need to verify that, a.s., 
	$Q^*_{(4)} \ll \theta$ 
	and under $Q^*$, properties (1)--(3) of Section~\ref{sec_kill:rate function LDP} are satisfied (with $\Theta$ replaced by $Q^*$).
	Without loss of generality assume that $(Q^n,\nu^n)$ converges weakly to $(Q^*,\nu^*)$ along the whole sequence.  
	Recall the canonical variables $(\bbdtil, \bbd, \rhobd, \sigmabdtil)$ and the canonical filtration $\{\Gmc_t\}$ introduced in Section~\ref{sec_kill:canonical}. 
	We can find a countable collection $\{\eta_j\}_{j=1}^{\infty}$ of continuous nonnegative functions with compact support in $\Rmb^d$, such that, denoting
	\begin{equation*}
		\rhobd^{\eta_j}_t(\xi) \doteq \int_{\Rd\times[0,t]} \eta_j(y) \, \rho(dy\,ds), \; t \in [0,T], \; j \in \Nmb, \; \xi = (\btil, b, \rho, \sigmatil) \in \Xi
	\end{equation*}
	and defining the stochastic process
	\begin{equation*}
		\Emc(t) \doteq (\bbdtil_t, \bbd_t, (\rhobd^{\eta_j}_t)_{j \in \Nmb}, \sigmabdtil), \; t \in [0,T]
	\end{equation*}
	with sample paths in $\Cmc_{\infty} \doteq \Cmb([0,T]:\Rmb^{\infty})$, we have $\Gmc_t = \sigma\{\Emc(\cdot \wedge t)\}$ for $t\in [0,T]$.

	We first verify (1). 
	It suffices to check that for all $0\le s \le t \le T$, $g \in \Cmb_b^2(\Rmb^d)$ and $f \in \Cmb_b(\Cmc_{\infty})$,
	\begin{equation} 
		\label{eq_kill:check1}
		\Ebf^*\left| \Ebf^{Q^*} f(\Emc(\cdot \wedge s)) \left(g(\bbd_t)-g(\bbd_s) - \frac{1}{2} \int_s^t \Delta g(\bbd_u) \, du \right) \right|^2 =0.
	\end{equation}
	Since $Q^n$ converges weakly to $Q^*$ the left side equals
	\begin{align*}
		& \lim_{n\to \infty} \Ebfbar \left| \Ebf^{Q^n} f(\Emc(\cdot \wedge s)) \left(g(\bbd_t)-g(\bbd_s) - \frac{1}{2} \int_s^t \Delta g(\bbd_u) \, du \right) \right|^2 \\
		& \quad = \lim_{n\to \infty} \frac{1}{n^2} \Ebfbar \left| \sum_{i=1}^n f(\Emc^n_i(\cdot \wedge s))\left( g(\betabd_i^n(t))- g(\betabd_i^n(s)) - \frac{1}{2} \int_s^t \Delta g(\betabd_i^n(u)) \, du\right)\right|^2,
	\end{align*}
	where $\Emc^n_i$ is defined similarly to $\Emc$ by replacing $(\bbdtil,\bbd,\rhobd,\sigmabdtil)$ with $(\betabdtil_i^n, \betabd_i^n, \rhobd_i^n, \Sbd_i^n)$.
	Conditioning on $\Fmchat_s$
	and using the fact that $\betabd^n$ is a standard $\Fmchat_t$-Brownian motion, we see that cross product terms do not contribute when the above squared sum is written as a double sum.
	So the above limit is $0$ which proves \eqref{eq_kill:check1}.
	
	Consider now (2).  
	It suffices to show that
	\begin{equation} 
		\label{eq_kill:check2_0}
		\Ebf^* \Ebf^{Q^*}\int_{\Rmb^d\times[0,T]} \|y\|^2 \, \rhobd(dy\,dt) < \infty.
	\end{equation}
	Note that $Q^n \Rightarrow Q^*$ as $n \to \infty$ and
	\begin{equation*}
		\Ebfbar \Ebf^{Q^n}\int_{\Rmb^d\times[0,T]} \|y\|^2\, \rhobd(dy\,dt) = \frac{1}{n} \sum_{i=1}^n \Ebfbar \int_0^T \|\psibd^n_i(t)\|^2 \,dt \le 2C_0.
	\end{equation*}
	Thus in order to prove \eqref{eq_kill:check2_0}, it suffices to show that the function
	\begin{equation}
		\label{eq_kill:check2}
		Q \mapsto \Ebf^Q\int_{\Rmb^d\times[0,T]} \|y\|^2 \, \rhobd(dy\,dt)
	\end{equation}
	is a lower semi-continuous function from $\Pmc(\Xi)$ to $[0,\infty]$.
	This is immediate from Fatou's lemma on observing that the function
	\begin{equation*}
		\xi = (\btil,b,\rho,\sigmatil) \mapsto \int_{\Rd \times [0,T]} \|y\|^2 \, \rho(dy\,ds)
	\end{equation*}
	is a lower semi-continuous function from $\Xi$ to $[0,\infty]$.
	This proves (2).
			
	Next we verify (3).  
	It suffices to show that for each $t \in [0,T]$
	\begin{equation} 
		\label{eq_kill:check3}
		\Ebf^* \Ebf^{Q^*} \left(\left| \ti \bbd_t - \bbd_t - \int_{\Rmb^d \times [0,t]} y \,\rhobd(dy\,ds) \right| \wedge 1 \right) = 0.
	\end{equation}
	Note that for each $t \in [0,T]$, the function
	\begin{equation}
		\label{eq_kill:check3_1}
		\xi = (\btil,b,\rho,\sigmatil) \mapsto \left| \btil_t - b_t - \int_{\Rmb^d \times [0,t]} y \,\rho(dy\,ds) \right| \wedge 1
	\end{equation}
	is a continuous and bounded function from $\Xi$ to $\Rmb_+$.
	Also, from \eqref{eq_kill:betabdtil^n},
	\begin{align*}
		& \Ebfbar \Ebf^{Q^n} \left(\left| \ti \bbd_t - \bbd_t - \int_{\Rmb^d \times [0,t]} y \,\rhobd(dy\,ds) \right| \wedge 1 \right) \\
		& \quad = \frac{1}{n} \sum_{i=1}^n \Ebfbar \left( \left | \betabdtil_i^n(t) - \betabd_i^n(t) - \int_0^t \psibd_i^n(s) \,ds \right| \wedge 1\right) = 0.
	\end{align*}
	Combining the above two observations and recalling that $Q^n \Rightarrow Q^*$ we have \eqref{eq_kill:check3}.
	
	Next we show that $Q^*_{(4)} \ll \theta$ a.s.
	By lower semi-continuity and convexity of the function $R(\cdot \| \theta)$, we have on using the first statement of the lemma that 
	\begin{equation}
		\label{eq_kill:check4}
		\Ebf^* R(Q^*_{(4)} \| \theta) = \Ebf^* R(\nu^* \| \theta) \le \liminf_{n \to \infty} \Ebfbar R(\nu^n \| \theta) \le \liminf_{n \to \infty} \Ebfbar \left[\frac{1}{n} \sum_{i=1}^n R(\nubd_i^n \| \theta)\right] \le C_0.
	\end{equation}
	So $R(Q^*_{(4)} \| \theta) < \infty$ and consequently $Q^*_{(4)} \ll \theta$ a.s.\ $\Pbd^*$.
	
	Finally, the third statement, namely $\Jmc(Q^*)<\infty$ a.s., follows immediately on combining \eqref{eq_kill:check2_0} and \eqref{eq_kill:check4}.
%
\end{proof}

For the following two lemmas, define
\begin{equation*}
	\|\betabdtil^n\|_{*,n}^2 \doteq \sup_{0 \le t \le T} \sum_{i=1}^n \|\betabdtil_i^n(t)\|^2.
\end{equation*}
From Cauchy--Schwarz and Doob's inequalities we have for all $n \in \mathbb{N}$
\begin{align}
	\frac{1}{n} \Ebfbar \|\betabdtil^n\|_{*,n}^2 & \le \frac{2}{n} \Ebfbar \sup_{0 \le t \le T} \sum_{i=1}^n \left( \|\betabd_i^n(t)\|^2 + \left\| \int_0^t \psibd_i^n(s)\,ds \right\|^2 \right) \notag \\
	& \le 8Td + \frac{2T}{n} \Ebfbar \sum_{i=1}^n \int_0^T \|\psibd_i^n(s)\|^2\,ds \notag \\
	& \le 8Td + 4TC_0 \doteq \Ctil_0, \label{eq_kill:beta_bd}
\end{align}
where the last line follows from \eqref{eq_kill:tightness_assumption}.
Using this moment bound we will now argue tightness of $\{\mutil^n\}$.

\begin{lemma}
	\label{lem_kill:tightness_2}
	Suppose \eqref{eq_kill:tightness_assumption} holds. 
	The sequence of random variables $\{\mutil^n\}_{n \ge 1}$ is tight in $\Dmc$.
\end{lemma}

\begin{proof}
	We first argue tightness of $\{\mutil^n(t)\}$ in $\Mmc(\Rd)$ for each $t \in [0,T]$.
	Note that $$G_4(m) \doteq \int_\Rd \|x\|^2 \, m(dx), \quad m \in \Mmc(\Rd)$$ is a tightness function on $\Mmc(\Rd)$.
	Also from \eqref{eq_kill:beta_bd} we have
	\begin{align*}
		\Ebfbar G_4(\mutil^n(t)) & = \Ebfbar \left[\frac{1}{n} \sum_{i=1}^n \|\betabdtil_i^n(t)\|^2 \one_{\left\{\Sbd_i^n > \int_0^t \lan\zeta,\,\mutil^n(s)\ran\, ds\right\}}\right] \le \frac{1}{n} \Ebfbar \|\betabdtil^n\|_{*,n}^2 \le \Ctil_0.
	\end{align*}
	Thus $\{\mutil^n(t)\}$ is tight in $\Mmc(\Rd)$ for each $t \in [0,T]$.
	
	Next we consider fluctuations of $\mutil^n$.  Recall that $\mutil^n$ is defined on some system $\Upsilon_{\Pi^n}$.
	For $\delta \in [0,T]$, let $\Tmc^{\delta,n}$ be the collection of all $[0,T-\delta]$-valued stopping times $\tau$ on $\Upsilon_{\Pi^n}$.
	In order to argue tightness of $\{\mutil^n\}$ in $\Dmc$, by Aldous-Kurtz tightness criterion (cf.~\cite[Theorem 2.7]{Kurtz1981approximation}), it suffices to show that
	\begin{equation}
		\label{eq_kill:tightness_2_fluctuation}
		\limsup_{\delta \to 0} \limsup_{n \to \infty} \sup_{\tau \in \Tmc^{\delta,n}} \Ebfbar d_{BL}(\mutil^n(\tau+\delta), \mutil^n(\tau)) = 0.
	\end{equation}
	Note that 
	\begin{align*}
		& d_{BL}(\mutil^n(\tau+\delta), \mutil^n(\tau)) \\
		& \:\:\:\:\:\: = \sup_{\|f\|_{BL} \le 1} | \lan f,\, \mutil^n(\tau+\delta) \ran - \lan f,\, \mutil^n(\tau) \ran | \\
		& \:\:\:\:\:\: \le \sup_{\|f\|_{BL} \le 1} \frac{1}{n} \sum_{i=1}^n \left|  f(\betabdtil_i^n(\tau+\delta)) \one_{\left\{\Sbd_i^n > \int_0^{\tau+\delta} \lan\zeta,\,\mutil^n(s)\ran\,ds\right\}} - f(\betabdtil_i^n(\tau)) \one_{\left\{\Sbd_i^n > \int_0^{\tau} \lan\zeta,\,\mutil^n(s)\ran\,ds\right\}} \right| \\
		& \:\:\:\:\:\: \le \sup_{\|f\|_{BL} \le 1} \frac{1}{n} \sum_{i=1}^n \left| f(\betabdtil_i^n(\tau+\delta)) - f(\betabdtil_i^n(\tau)) \right| \one_{\left\{\Sbd_i^n > \int_0^{\tau+\delta} \lan\zeta,\,\mutil^n(s)\ran\,ds\right\}} \\
		& \qquad + \sup_{\|f\|_{BL} \le 1} \frac{1}{n} \sum_{i=1}^n \left| f(\betabdtil_i^n(\tau)) \left( \one_{\left\{\Sbd_i^n > \int_0^{\tau+\delta} \lan\zeta,\,\mutil^n(s)\ran\,ds\right\}} - \one_{\left\{\Sbd_i^n > \int_0^{\tau} \lan\zeta,\,\mutil^n(s)\ran\,ds\right\}} \right) \right| \\
		& \:\:\:\:\:\: \doteq \Tmc_1^n + \Tmc_2^n.
	\end{align*}
	It then suffices to show that for $j=1,2$,
	\begin{equation}
		\label{eq_kill:tightness_T}
		\limsup_{\delta \to 0} \limsup_{n \to \infty} \sup_{\tau \in \Tmc^{\delta,n}} \Ebfbar \Tmc_j^n = 0.
	\end{equation}	
	
	For $\Tmc_1^n$, we have
	\begin{align*}
		\Tmc_1^n & \le \frac{1}{n} \sum_{i=1}^n \left\| \betabdtil_i^n(\tau+\delta) - \betabdtil_i^n(\tau) \right\| \\
		& \le \frac{1}{n} \sum_{i=1}^n \left( \|\betabd_i^n(\tau+\delta) - \betabd_i^n(\tau)\| + \int_{\tau}^{\tau+\delta} \|\psibd_i^n(s)\| \, ds \right).
	\end{align*}
	So by Cauchy--Schwarz inequality,
	\begin{align*}
		& \sup_{\tau \in \Tmc^{\delta,n}} \Ebfbar \Tmc_1^n \\
		& \quad \le \sup_{\tau \in \Tmc^{\delta,n}} \left( \Ebfbar \frac{1}{n} \sum_{i=1}^n \left\| \betabd_i^n(\tau+\delta) - \betabd_i^n(\tau) \right\|^2 \right)^\half + \sqrt{\delta} \left( \Ebf \frac{1}{n} \sum_{i=1}^n \int_0^T \|\psibd_i^n(s)\|^2 \, ds \right)^\half \\
		& \quad \le \sqrt{d\delta} + \sqrt{2C_0\delta},
	\end{align*}	
	which implies that \eqref{eq_kill:tightness_T} holds for $j=1$.
	
	Now consider $\Tmc^n_2$.
	Fix $K \in (0,\infty)$.
	Since $\|f\|_\infty \le 1$,
	\begin{align*}
		\Tmc^n_2 & \le \frac{1}{n} \sum_{i=1}^n \one_{\left\{\int_0^{\tau} \lan\zeta,\,\mutil^n(s)\ran\,ds < \Sbd_i^n \le \int_0^{\tau+\delta} \lan\zeta,\,\mutil^n(s)\ran\,ds\right\}} \\
		& \le \one_{\{ \|\betabdtil^n\|_{*,n}^2 > nK \}} + \one_{\{ \|\betabdtil^n\|_{*,n}^2 \le nK \}} Q^n_{(4)} \left( \left( \int_0^{\tau} \lan\zeta,\,\mutil^n(s)\ran\,ds, \int_0^{\tau+\delta} \lan\zeta,\,\mutil^n(s)\ran\,ds \right] \right) \\
		& \doteq \Tmc^n_{2,1} + \Tmc^n_{2,2}.
	\end{align*}
	For $\Tmc^n_{2,1}$, from \eqref{eq_kill:beta_bd} it follows that
	\begin{equation}
	\label{eq_kill:T_2_1}
		\Ebfbar \Tmc^n_{2,1} \le \frac{\Ebfbar \|\betabdtil^n\|_{*,n}^2}{nK} \le \frac{\Ctil_0}{K}.
	\end{equation}
	For $\Tmc^n_{2,2}$, note that on the set $\{ \|\betabdtil^n\|_{*,n}^2 \le nK \}$, we have from \eqref{eq_kill:zeta_2} that for $t \in [0,T]$,
	\begin{align}
		\lan \zeta,\,\mutil^n(t) \ran &\le \frac{1}{n} \sum_{i=1}^n \zeta(\betabdtil^n_i(t)) \le \frac{1}{n} \sum_{i=1}^n \Ctil_\zeta (1 + \|\betabdtil^n_i(t)\|^2) \le \Ctil_\zeta + \frac{\Ctil_\zeta \|\betabdtil^n\|_{*,n}^2}{n} \notag\\
		&\le \Ctil_\zeta (1+K), \label{eq:eq1144}
	\end{align}
	and hence
	\begin{align*}
		\left| \int_0^{\tau+\delta} \lan\zeta,\,\mutil^n(s)\ran\,ds - \int_0^{\tau} \lan\zeta,\,\mutil^n(s)\ran\,ds \right| \le C_K \delta
	\end{align*}
	with $C_K = \Ctil_\zeta (1+K)$.
	From this it follows that
	\begin{equation*}
		\Tmc^n_{2,2} \le \one_{\{ \|\betabdtil^n\|_{*,n}^2 \le nK \}} \max_{j \in \Nmb_0} Q^n_{(4)}([j C_K \delta, (j+2) C_K \delta]).
	\end{equation*}
	Using this we have
	\begin{equation}
		\label{eq_kill:tightness_T_2}
		\sup_{\tau \in \Tmc^{\delta,n}} \Ebfbar \Tmc^n_{2,2} \le \Ebfbar \max_{j \in \Nmb_0} Q^n_{(4)}([j (C_K \delta, (j+2) C_K \delta]).
	\end{equation}
	By Lemmas~\ref{lem_kill:tightness_1} and~\ref{lem_kill:char_limit_1}, we can assume without loss of generality that $(Q^n,\nu^n)$ converges weakly along the whole sequence to $(Q^*,\nu^*)$, given on some probability space $(\Omega^*, \Fmc^*, \Pbd^*)$, with $Q^* \in \Pmc_\infty$ a.s.\ $\Pbd^*$.
	Note that if $\{m_n\}, m$ are probability measures on $[0, \infty)$ such that $m$ is absolutely continuous with respect to the Lebesgue measure, $m_n\to m$ as $n\to \infty$, and $\delta, c >0$, then
	\begin{align}
		\max_{j\in \Nmb_0} m_n ([j\delta, j\delta +c]) \to 	\max_{j\in \Nmb_0} m ([j\delta, j\delta +c]). \label{eq:eq1147}
	\end{align}
	It then follows from \eqref{eq_kill:tightness_T_2}, dominated convergence theorem, \eqref{eq:eq1147} and Remark \ref{rmk_kill:Pmc_infty} that 
	\begin{align*}
		\limsup_{n \to \infty} \sup_{\tau \in \Tmc^{\delta,n}} \Ebfbar \Tmc_2^n & \le \limsup_{n \to \infty} \Ebfbar \max_{j \in \Nmb_0} Q^n_{(4)}([j C_K \delta, (j+2) C_K \delta]) \\
		&= \Ebf^* \max_{j \in \Nmb_0} Q^*_{(4)}([j C_K \delta, (j+2) C_K \delta]).
	\end{align*}
	Using \eqref{eq_kill:check4}, for every $M>1$,
	\begin{equation}\Ebf^* \max_{j \in \Nmb_0} Q^*_{(4)}([j C_K \delta, (j+2) C_K \delta]) \le 2M C_K \delta + \frac{C_0}{\log M}. \label{eq:eq1209} \end{equation}
		This shows that the left side in the above equation converges to $0$ as $\delta \to 0$ and $M \to \infty$.
	Combining this convergence with \eqref{eq_kill:T_2_1} and then sending $K \to \infty$ implies that \eqref{eq_kill:tightness_T} also holds for $j=2$ which completes the proof of tightness of $\{\mutil^n\}$.
\end{proof}

The following lemma gives an important characterization of weak limit points of $\{\mutil^n\}$.
Recall the definition of $\varpi_\Theta$ for $\Theta \in \Pmc_\infty^1$ given in Section~\ref{sec_kill:rate function LDP}.

\begin{lemma} 
	\label{lem_kill:char_limit_2}
	Suppose \eqref{eq_kill:tightness_assumption} holds and $(Q^n,\nu^n,\mutil^n)$ converges along a subsequence, in distribution, to $(Q^*,\nu^*,\mutil^*)$ given on some probability space $(\Omega^*,\Fmc^*,\Pbd^*)$.
	Then $\Pbd^*$-a.s., $Q^* \in \Pmc_\infty^1$ and $\mutil^* = \varpi_{Q^*}$, i.e.\
	\begin{equation}
		\label{eq_kill:char_limit_2}
		\lan f,\,\mutil^*(t) \ran = \Ebf^{Q^*} \left[ f(\bbdtil_t) \one_{\left\{\sigmabdtil > \int_0^t \lan\zeta,\,\mutil^*(s)\ran \, ds \right\}} \right], \; \forall f \in \Cmb_b(\Rd), t \in [0,T].
	\end{equation}
\end{lemma}

\begin{proof}

	Since $R(\Pi^n \| \theta^{\otimes n}) < \infty$, we have that a.s., $\{\Sbd^n_i\}_{i=1}^n$ are distinct and hence, a.s., the total mass of $\mutil^n(t)$ decreases by at most $1/n$ at any time instant $t$.
	Consequently, $0 \le \mutil^n(t-)(A) - \mutil^n(t)(A) \le 1/n$ for all $A \in \Bmc(\Rd)$ and $t\in [0,T]$, a.s.
	Also, without loss of generality assume that $(Q^n,\nu^n,\mutil^n)$ converges weakly to $(Q^*,\nu^*,\mutil^*)$ along the whole sequence. 
	We claim that, jointly with the above collection,
	\begin{equation}
	\label{eq_kill:char_limit_2_claim}
		\int_0^{\cdot} \lan \zeta,\, \mutil^n(s) \ran\,ds \Rightarrow \int_0^{\cdot} \lan \zeta,\, \mutil^*(s) \ran\,ds \mbox{ in } \Cmb([0,T]:\Rmb),
	\end{equation}
	as $n \to \infty$.
	Suppose for now that the claim holds.
	By appealing to Skorokhod representation we can assume without loss of generality that $(Q^n,\nu^n,\mutil^n, 	\int_0^{\cdot} \lan \zeta,\, \mutil^n(s) \ran\,ds)$ convergences a.s.\ on $(\Omega^*,\Fmc^*,\Pbd^*)$.		
	From this and Lemma \ref{lem_kill:char_limit_1} we can find  $\Nmc \in \Fmc^*$ such that $\Pbd^*(\Nmc)=0$ and on $\Nmc^c$, we have that $(Q^n,\nu^n,\mutil^n, \int_0^{\cdot} \lan \zeta,\, \mutil^n(s) \ran\,ds) \to (Q^*,\nu^*,\mutil^*, \int_0^{\cdot} \lan \zeta,\, \mutil^*(s) \ran\,ds)$, $Q^*_{(4)}=\nu^*$, $Q^* \in \Pmc_\infty$, $\Jmc(Q^*)<\infty$ and $0 \le \mutil^n(t-)(A) - \mutil^n(t)(A) \le 1/n$ for all $A \in \Bmc(\Rd)$ and $t\in [0,T]$.
	The last statement implies that on $\Nmc^c$
	\begin{equation*}
		\sup_{t \in (0,1]} d_{BL}(\mutil^n(t),\mutil^n(t-)) = \sup_{t \in (0,1]} \sup_{\|f\|_{BL} \le 1} | \lan f,\mutil^n(t) \ran - \lan f,\mutil^n(t-) \ran | \le \frac{1}{n} \to 0
	\end{equation*}
	as $n \to \infty$ and hence $\mutil^* \in \Cmb([0,T]:\Mmc(\Rd))$ on $\Nmc^c$.
	Fix $f \in \Cmb_b(\Rd)$ and $t \in [0,T]$.
	Then, on $\Nmc^c$,  $\lan f,\,\mutil^n(t) \ran \to \lan f,\,\mutil^*(t) \ran$. 
	Define for $t\in [0,T]$, $g_t \colon \Pmc(\Xi) \times \Dmc \to \Rmb$ as
	\begin{equation}
		\label{eq_kill:char_limit_2_g}
		g_t(\Theta,\pi) = \Ebf^\Theta \left[ f(\bbdtil_t) \one_{\left\{\sigmabdtil > \int_0^t \lan\zeta,\,\pi(s)\ran \, ds \right\}} \right].
	\end{equation}	
	From \eqref{eq_kill:betabdtil^n} and \eqref{eq_kill:mutil^n} we have
		$\lan f,\,\mutil^n(t) \ran = g_t(Q^n,\mutil^n)$.

	Noting that the RHS of \eqref{eq_kill:char_limit_2} is $g_t(Q^*,\mutil^*)$, it then suffices to show that as $n \to \infty$, $g_t(Q^n,\mutil^n) \to g_t(Q^*,\mutil^*)$ in probability.
Note that
	\begin{equation}
	\label{eq_kill:char_limit_2_g0}
		|g_t(Q^n,\mutil^n) - g_t(Q^*,\mutil^*)| \le |g_t(Q^n,\mutil^n) - g_t(Q^n,\mutil^*)| + |g_t(Q^n,\mutil^*) - g_t(Q^*,\mutil^*)|.
	\end{equation}
	Since on $\Nmc^c$, $Q^* \in \Pmc_{\infty}$, by Remark~\ref{rmk_kill:Pmc_infty} we have on this set,
	\begin{equation}
		\label{eq_kill:char_limit_2_g1}
		\lim_{n \to \infty} |g_t(Q^n,\mutil^*) - g_t(Q^*,\mutil^*)| = 0.
	\end{equation}
	For the first term on the right hand side of \eqref{eq_kill:char_limit_2_g0}, we have, on $\Nmc^c$,
	\begin{align*}
		& |g_t(Q^n,\mutil^n) - g_t(Q^n,\mutil^*)| \\
		& \quad \le \Ebf^{Q^n} \left| f(\bbdtil_t) \left( \one_{\left\{\sigmabdtil > \int_0^t \lan\zeta,\,\mutil^n(s)\ran \, ds \right\}} - \one_{\left\{\sigmabdtil > \int_0^t \lan\zeta,\,\mutil^*(s)\ran \, ds \right\}} \right) \right| \\
		& \quad \le \|f\|_\infty Q^n \left( \sigmabdtil \mbox{ lies between } \int_0^t \lan\zeta,\,\mutil^n(s)\ran \, ds \mbox{ and } \int_0^t \lan\zeta,\,\mutil^*(s)\ran \, ds\right).
	\end{align*}
	From convergence of $\int_0^{\cdot} \lan \zeta,\, \mutil^n(s) \ran\,ds$ to $\int_0^{\cdot} \lan\zeta,\,\mutil^*(s)\ran \, ds$ and \eqref{eq:eq1144},
	it follows that for any $\delta > 0$,
	\begin{align*}
		& \limsup_{n \to \infty} |g_t(Q^n,\mutil^n) - g_t(Q^n,\mutil^*)| \\
		& \quad \le \|f\|_\infty \limsup_{n \to \infty} Q^n \left( \sigmabdtil \mbox{ lies between } \int_0^t \lan\zeta,\,\mutil^n(s)\ran \, ds \mbox{ and } \int_0^t \lan\zeta,\,\mutil^*(s)\ran \, ds\right) \\
		& \quad \le \|f\|_\infty \limsup_{n \to \infty} \max_{j \in \Nmb_0} Q^n_{(4)}([j\delta,(j+1)\delta]) \\
		& \quad = \|f\|_\infty \max_{j \in \Nmb_0} Q^*_{(4)}([j\delta,(j+1)\delta]),
	\end{align*}
	where the last equality uses \eqref{eq:eq1147}.
	Using a similar bound as in \eqref{eq:eq1209} and letting $\delta \to 0$ in the above display gives 
	 $\limsup_{n \to \infty} |g_t(Q^n,\mutil^n) - g_t(Q^n,\mutil^*)| = 0$ on $\Nmc^c$. 
	This together with \eqref{eq_kill:char_limit_2_g0} and \eqref{eq_kill:char_limit_2_g1} proves that on $\Nmc^c$, for all $t \in [0,T]$, $g_t(Q^n,\mutil^n) \to g_t(Q^*,\mutil^*)$ as $n \to \infty$ and hence $ \lan f,\,\mutil^*(t) \ran =  g_t(Q^*,\mutil^*)$. Therefore, $\mutil^* = \varpi_{Q^*}$ a.s.
	
	Now we prove the claim \eqref{eq_kill:char_limit_2_claim}.
	Fix $K \in (0,\infty)$ and define function $\zeta_K(x) \doteq \min \{ \zeta(x), K \}$ for $x \in \Rd$.
	Clearly as $n \to \infty$,
	\begin{equation}
	\label{eq_kill:char_limit_2_claim_1}
		\int_0^{\cdot} \lan \zeta_K,\, \mutil^n(s) \ran\,ds \Rightarrow \int_0^{\cdot} \lan \zeta_K,\, \mutil^*(s) \ran\,ds \mbox{ in } \Cmb([0,T]: \Rmb).
	\end{equation}	
	From \eqref{eq_kill:zeta} it follows that for all $x \in \Rd$,
	\begin{align}
		\label{eq_kill:zeta_K}
		\begin{aligned}
		\zeta(x) - \zeta_K(x) & \le \zeta(x) \one_{\{ \zeta(x) > K \}} \le \zeta(x) \left( \frac{\zeta(x)}{K} \right)^{2/p-1} \\
		& \le \frac{C_\zeta^{2/p} (1+\|x\|^p)^{2/p}}{K^{2/p-1}} \le \frac{\kappa_1(1+\|x\|^2)}{K^{2/p-1}}.
		\end{aligned}
	\end{align}
	So using \eqref{eq_kill:mutil^n} we have
	\begin{align*}
		& \sup_{n \in \Nmb} \sup_{0\le t \le T}\Ebfbar \left| \int_0^t \lan \zeta,\, \mutil^n(s) \ran\,ds - \int_0^t \lan \zeta_K,\, \mutil^n(s) \ran\,ds \right| \\
		& \quad \le \sup_{n \in \Nmb} \int_0^T \Ebfbar \lan \zeta - \zeta_K,\, \mutil^n(s) \ran\,ds \\
		& \quad \le \sup_{n \in \Nmb} \int_0^T \left(\Ebfbar \frac{1}{n} \sum_{i=1}^n \frac{\kappa_1(1+\|\betabdtil^n_i(s)\|^2)}{K^{2/p-1}}\right) \, ds \\
		& \quad \le \sup_{n \in \Nmb} \frac{\kappa_1}{K^{2/p-1}} \left( T + \frac{T}{n} \Ebfbar \|\betabdtil\|_{*,n}^2 \right) \\
		& \quad \le \frac{\kappa_2}{K^{2/p-1}} \\
		& \quad \to 0
	\end{align*}
	as $K \to \infty$, where the last inequality follows from \eqref{eq_kill:beta_bd}.
	Also by Fatou's lemma,
	\begin{align*}
		\Ebf^* \left| \int_0^t \lan \zeta,\, \mutil^*(s) \ran\,ds - \int_0^t \lan \zeta_K,\, \mutil^*(s) \ran\,ds \right| & 
		 \le \liminf_{n \to \infty} \int_0^t \Ebfbar \lan \zeta - \zeta_K,\, \mutil^n(s) \ran\,ds \\
		& \le \frac{\kappa_2}{K^{2/p-1}} \\
		& \to 0
	\end{align*}
	as $K \to \infty$.
	The claim \eqref{eq_kill:char_limit_2_claim} then follows on combining above two displays with \eqref{eq_kill:char_limit_2_claim_1}.
\end{proof}


\section{Laplace upper bound} 
\label{sec_kill:Lapuppbd}

In this section we prove the \textit{Laplace upper bound} \eqref{eq_kill:Lapuppbd}:
For every $F \in \Cmb_b(\Dmc)$,
\begin{equation*}
	\liminf_{n\to \infty}-\frac{1}{n}\log \Ebfbar \big[\exp\bigl(-n F(\mu^{n})\bigr)\big]	\ge
	\inf_{\pi \in \Dmc}\{ F(\pi) + I(\pi) \}.
\end{equation*}
Let $\eps \in (0,1)$ be arbitrary.  
For each $n \in \Nmb$, let $\Pi^n \in \Pmc(\Rmb_+^n)$, $\Upsilon_{\Pi^n}$ (with the associated  $\Sbd^n$) and $\psibd^n \in \Amc_n(\Upsilon_{\Pi^n})$ be an $\eps$-optimal control in \eqref{eq_kill:rep_key}, namely
\begin{align} 
	\label{eq_kill:eps_optimal_lower}
	\begin{aligned}
		& -\frac{1}{n}\log \Ebf \left[\exp\left(-n F(\mu^{n})\right)\right] \\
		& \qquad \ge \Ebfbar \left[ \frac{1}{n} \sum_{i=1}^n \left( R(\nubd_i^n \| \theta) + \frac{1}{2} \int_0^T \|\psibd_i^n(s)\|^2 \, ds \right) + F(\mutil^n) \right] - \eps,
	\end{aligned}
\end{align}
where $(\nubd_i^n)_{i=1}^n$ and $\mutil^n$ are defined in \eqref{eq_kill:nubd^n} and \eqref{eq_kill:mutil^n}.  
Then with $Q^n, \nu^n$ defined as in \eqref{eq_kill:Q^n_general} and \eqref{eq_kill:nu^n_general},
we have from  Jensen's inequality that
\begin{align*}
	& -\frac{1}{n}\log \Ebf \big[\exp\bigl(-n F(\mu^{n})\bigr)\big] \\
	& \qquad \ge \Ebfbar \left[ R(\nu^n \| \theta) + \Ebf^{Q^n} \left(\frac{1}{2} \int_{\Rd \times [0,T]} \|y\|^2 \, \rhobd(dy\,ds) \right) + F(\mutil^n) \right] - \eps.
\end{align*}
The  inequality \eqref{eq_kill:eps_optimal_lower} implies that
\begin{equation*} 
	\sup_{n \in \Nmb} \Ebfbar \left[ \frac{1}{n} \sum_{i=1}^n \left( R(\nubd_i^n \| \theta) + \frac{1}{2} \int_0^T \|\psibd_i^n(s)\|^2 \, ds \right) \right] \le 2 \|F\|_{\infty} + 1,
\end{equation*}
i.e.\  \eqref{eq_kill:tightness_assumption} holds with $C_0 = 2 \|F\|_{\infty} + 1$.

Thus from Lemmas~\ref{lem_kill:tightness_1} and~\ref{lem_kill:tightness_2} we have that $\{(Q^n,\nu^n,\mutil^n)\}$ is tight.
Assume without loss of generality that $(Q^n,\nu^n,\mutil^n)$ converges along the whole sequence weakly to $(Q^*,\nu^*,\mutil^*)$, given on some probability space $(\Omega^*,\Fmc^*,\Pbd^*)$.
By Lemmas~\ref{lem_kill:char_limit_1} and~\ref{lem_kill:char_limit_2}, we have $Q^* \in \Pmc_\infty^1$, $Q^*_{(4)} = \nu^*$, and $\varpi_{Q^*} = \mutil^*$ a.s.\ $\Pbd^*$.
Using the lower semi-continuity of the function $R(\cdot\|\theta)$ and the map in \eqref{eq_kill:check2}, we now have from Fatou's lemma that
\begin{align*}
	& \liminf_{n\to \infty} \Ebfbar \left( R(\nu^n \| \theta) + \Ebf^{Q^n} \left[ \half \int_{\Rmb^d \times [0,T]} \|y\|^2 \,\rhobd(dy\,ds) \right] \right) \\
	& \quad \ge \Ebf^* \left( R(\nu^* \| \theta) + \Ebf^{Q^*} \left[ \half \int_{\Rmb^d \times [0,T]} \|y\|^2 \,\rhobd(dy\,ds) \right] \right) \\
	& \quad = \Ebf^* (\Jmc(Q^*)) \\
	& \quad \ge \Ebf^* (I(\varpi_{Q^*})).
\end{align*}
Also since $F \in \Cmb_b(\Dmc)$,
\begin{equation*}
	\lim_{n \to \infty} \Ebfbar F(\mutil^n) = \Ebf^* F(\mutil^*) = \Ebf^* F(\varpi_{Q^*}).
\end{equation*}
Combining above three displays we have
\begin{align*}
	\liminf_{n\to \infty} -\frac{1}{n}\log \Ebfbar \left[\exp\left(-n F(\mu^{n})\right)\right] & \ge \Ebf^* \left(F(\varpi_{Q^*}) + I(\varpi_{Q^*})\right) - \eps \\
	& \ge \inf_{\pi \in \Dmc} \left(F(\pi) + I(\pi)\right) - \eps.
\end{align*}
Since $\eps>0$ is arbitrary, this completes the proof of the Laplace upper bound. \qed

\section{Laplace lower bound}
\label{sec_kill:Laplowbd}

In this section we prove the \textit{Laplace lower bound} \eqref{eq_kill:Laplowbd}: 
For every $F \in \Cmb_b(\Dmc)$,
\begin{equation*}
	\limsup_{n\to \infty}-\frac{1}{n}\log \Ebfbar \left[\exp\left(-n F(\mu^{n})\right)\right]	\le
	\inf_{\pi \in \Dmc}\{ F(\pi) + I(\pi) \}.
\end{equation*}
Fix $F \in \Cmb_b(\Dmc)$ and let $\eps >0$ be arbitrary and let $\pi^* \in \Dmc$ be such that
\begin{equation*}
	F(\pi^*) + I(\pi^*) \le \inf_{\pi \in \Dmc} \{F(\pi) + I(\pi)\} + \veps.
\end{equation*}
Note in particular that $I(\pi^*)<\infty$. 
Next let $\Theta \in \Pmc_{\infty}^1$ be such that $\varpi_{\Theta} = \pi^*$ and
	$\Jmc(\Theta) \le I(\pi^*) + \veps$.
Recall the canonical variables $(\bbdtil, \bbd, \rhobd, \sigmabdtil)$ introduced in
\eqref{eq:eq752}.
We claim that without loss of generality one can assume that
\begin{equation} 
	\label{eq_kill:rhobd_claim}
	\rhobd(A \times [0,t]) = \int_0^t \delta_A(\psibd(s)) \, ds, \; t \in [0,T], A \in \Bmc(\Rd), \text{ a.s.\ } \Theta,
\end{equation}
where $\psibd$ is an $\Rd$-valued $\Gmc_t$-progressively measurable process such that
\begin{equation*}
	\Ebf^\Theta \int_0^T \|\psibd(s)\|^2 \, ds < \infty.
\end{equation*}
To see this, define $\psibd$ and $\rhobdtil$ on $\Xi$ as follows.
For $t \in [0,T]$, $A \in \Bmc(\Rd)$ and $\xi = (\btil,b,\rho,\sigmatil) \in \Xi$,
\begin{equation*}
	\psibd[\xi](t) \doteq \int_\Rd y \, \rho(dy | t), \: \rhobdtil[\xi](A \times [0,t]) \doteq \int_0^t \delta_A(\psibd[\xi](s)) \, ds,
\end{equation*}
where $\rho(dy\,dt) = \rho(dy|t)\,dt$.
Then
\begin{equation} 
	\label{eq_kill:rhobdtil_1}
	\int_{\Rd \times [0,t]} y \,\rhobdtil(dy\,ds) = \int_0^t \psibd(s) \, ds = \int_{\Rd \times [0,t]} y \,\rhobd(dy\,ds)
\end{equation}
and
\begin{equation}
	\label{eq_kill:rhobdtil_2}
	\int_{\Rd \times [0,t]} \|y\|^2 \,\rhobdtil(dy\,ds) = \int_0^t \|\psibd(s)\|^2 \, ds \le \int_{\Rd \times [0,t]} \|y\|^2 \,\rhobd(dy\,ds).
\end{equation}
Thus with $\Thetatil \doteq \Theta \circ (\bbdtil,\bbd,\rhobdtil,\sigmabdtil)^{-1}$, it follows from \eqref{eq_kill:rhobdtil_1} and \eqref{eq_kill:rhobdtil_2} that $\Thetatil \in \Pmc_\infty^1$, $\varpi_{\Thetatil} = \varpi_\Theta = \pi^*$ and $\Jmc(\Thetatil) \le \Jmc(\Theta) \le I(\pi^*) + \veps$.
This proves the claim.
Henceforth we assume \eqref{eq_kill:rhobd_claim}.

We now construct a filtered probability space $\Upsilon \doteq (\bar \Omega, \bar \Fmc, \{\hat \Fmc_t\}_{0\le t \le T}, \bar \Pbd)$ as follows.
Let	$\bar \Omega \doteq \Xi^{\otimes \infty}$, $\bar \Fmc$ be the $\bar \Pbd$-completion of $\Bmc(\Xi^{\otimes \infty})$, where $\bar \Pbd \doteq \Theta^{\otimes \infty}$.
Define canonical variables on this space as follows: For $\xi \doteq (\xi_i)_{i\in \Nmb} \in \bar \Omega$, where $\xi_i \doteq (\btil_i, b_i, \rho_i, \sigmatil_i)$
\begin{equation*}
	\ti \bbd_i(\xi) \doteq \ti b_i ,\; \bbd_i(\xi) \doteq b_i, \; \rhobd_i(\xi) \doteq \rho_i, \; \sigmabdtil_i(\xi) = \sigmatil_i, \; i \in \Nmb.
\end{equation*}
Let
	$\hat \Fmc_t$ be the $\bar \Pbd$-completion of $\sigma \{ \ti \bbd_i(s), \bbd_i(s), \rhobd_i(A \times [0,s]), \sigmabdtil_i, i \in \Nmb, s \le t, A \in \Bmc(\Rmb^d) \}$.
Define
\begin{equation*}
	\psibd_i(t) \doteq \int_{\Rmb^d} y \, \rhobd_i(dy| t), \; i \in \Nmb, t \in [0,T].
\end{equation*}
Note that by \eqref{eq_kill:rhobd_claim}, a.s.\ $\Pbdbar$, $\rhobd_i(A \times [0,t]) = \int_0^t \delta_A(\psibd_i(s)) \, ds$ for $A \in \Bmc(\Rd)$, $t \in [0,T]$ and $i \in \Nmb$.
With above choices of $\Upsilon$ and canonical random variables, let $(\nubd_i^n)_{i=1}^n$ and $\mutil^n$ be defined as above Lemma~\ref{lem_kill:rep}, with $\Pi^n = \Theta_{(4)}^{\otimes n}$ and $(\Sbd^n, \betabdtil^n, \betabd^n)$ replaced with $(\sigmabdtil^n, \bbdtil^n, \bbd^n)$, where $\sigmabdtil^n \doteq (\sigmabdtil_i)_{i=1}^n$, $\bbdtil^n \doteq (\bbdtil_i)_{i=1}^n$ and $\bbd^n \doteq (\bbd_i)_{i=1}^n$.
Note in particular that $\nubd_i^n = \Theta_{(4)}$ for each $i = 1,\dotsc,n$.
It follows from Lemma~\ref{lem_kill:rep} that
\begin{equation}
	\label{eq_kill:repnredulbd}
	-\frac{1}{n} \log  \Ebf e^{-n F(\mu^n)}
	\le  \bar 
	\Ebf\left[ \frac{1}{n} \sum_{i=1}^n \left( R(\nubd_i^n\|\theta) + \half \int_0^T \|\psibd_i(s)\|^2 \, ds \right) + F(\mutil^n) \right].
\end{equation}
Note that
\begin{align}
	& \Ebfbar \left[ \frac{1}{n} \sum_{i=1}^n \left( R(\nubd_i^n\|\theta) + \half \int_0^T \|\psibd_i(s)\|^2 \, ds \right) \right] \notag \\
	& \quad = R(\Theta_{(4)} \| \theta) + \frac{1}{2n} \sum_{i=1}^n \Ebfbar \left[ \int_{\Rd \times [0,T]} \|y\|^2 \, \rhobd_i(dy\,ds) \right] \notag \\
	& \quad = R(\Theta_{(4)} \| \theta) + \Ebf^{\Theta} \left[ \half \int_{\Rd \times [0,T]} \|y\|^2 \, \rhobd(dy\,ds) \right] \notag \\
	& \quad = \Jmc(\Theta), \label{eq_kill:complete_varuppbd}
\end{align}
where the second equality holds since $\rhobd_i$ are i.i.d.\ under $\Pbdbar$.
The above calculation shows that \eqref{eq_kill:tightness_assumption} holds with $C_0$ replaced by $\Jmc(\Theta) < \infty$.
Define $(Q^n,\nu^n)$ as in \eqref{eq_kill:Q^n_general} and \eqref{eq_kill:nu^n_general} (with $(\Sbd^n, \betabdtil^n, \betabd^n)$ replaced by $(\sigmabdtil^n, \bbdtil^n, \bbd^n)$ and $(\rhobd_i^n)_{i=1}^n$ replaced by $(\rhobd_i)_{i=1}^n$). 
It follows from Lemmas~\ref{lem_kill:tightness_1} and~\ref{lem_kill:tightness_2} that  $\{(Q^n,\nu^n,\mutil^n)\}$ is a tight sequence of $\Pmc(\Xi) \times \Pmc(\Rd) \times \Dmc$-valued random variables.
Suppose $(Q^*,\nu^*,\mutil^*)$ is a weak limit point of the sequence given on some probability space $(\Omega^*, \Fmc^*, \Pbd^*)$. 
From Lemmas~\ref{lem_kill:char_limit_1} and~\ref{lem_kill:char_limit_2} it follows that $Q^* \in \Pmc_{\infty}$, $Q^*_{(4)} = \nu^*$ and $\mutil^* = \varpi_{Q^*}$ a.s.\ $\Pbd^*$.
By law of large numbers $Q^* = \Theta$ a.s.\ $\Pbd^*$ and hence we have $\mutil^* = \varpi_{\Theta} = \pi^*$ a.s.\ $\Pbd^*$.
Combining above observations with \eqref{eq_kill:repnredulbd} and \eqref{eq_kill:complete_varuppbd} we have
\begin{align*}
	\limsup_{n\to \infty} -\frac{1}{n} \log \Ebf e^{-n F(\mu^n)} 
	& \le \Jmc(\Theta) + F(\pi^*) \\
	& \le I(\pi^*) + F(\pi^*) + \eps \\
	& \le \inf_{\pi \in \Dmc} \{F(\pi) + I(\pi)\} + 2\eps,
\end{align*}
where the last two inequalities follow from the choices of $\Theta$ and $\pi^*$.
Since $\veps>0$ is arbitrary, the desired Laplace lower bound follows.\qed

\section{$I$ is a rate function. }
\label{sec_kill:rate_function_pf}

In this section we will prove that the function $I$ defined in \eqref{eq_kill:rate_function} is a rate function, namely we prove the following. 

\begin{proposition}
\label{prop_kill:rate_function}
For each fixed $M < \infty$, the set $F_M \doteq \{\pi \in \Dmc: I(\pi) \le M\}$ is compact.
\end{proposition}

\begin{proof}
Let $\{\pi_n\}_{n\in \Nmb} \subset F_M$.
Then for each $n \in \Nmb$ there exists $\Theta^n \in \Pmc_{\infty}^1$ such that $\varpi_{\Theta^n} = \pi_n$ and $\Jmc(\Theta^n) \le M+ \frac{1}{n}$.
In particular
\begin{equation*}
	\sup_{n\in \Nmb} \left( R(\Theta^n_{(4)} \| \theta) + \Ebf^{\Theta^n}\left[ \half \int_{\Rd\times [0,T]} \|y\|^2 \,\rhobd(dy\,ds) \right] \right) \le M+1.
\end{equation*}

Using the above bound we now argue that $\{\Theta^n\}$ is tight.
The proof is similar to that of Lemma~\ref{lem_kill:tightness_1} and so we only provide a sketch.
Note that $\Theta^n_{(2)}$ is the $d$-dimensional Wiener measure for each $n$ and so the tightness of $\{\Theta^n_{(2)}\}$ is immediate.
Next, recall tightness functions $G_j$ for $j=1$, $2$, $3$ and the function $g_3$ in the proof of Lemma~\ref{lem_kill:tightness_1}.
Since $G_1(\Theta^n_{(3)}) = \Ebf^{\Theta^n} \int_{\Rd \times [0,T]} \|y\|^2 \, \rhobd(dy\,ds) \le 2(M+1)$, $\{\Theta^n_{(3)}\}$ is tight.
To argue tightness of $\{\Theta^n_{(1)}\}$, define $\ubd \colon \Xi \to \Cmc$ as $\ubd(t)[\xi] \doteq \int_{\Rd \times [0,t]} y \, \rhobd(dy\,ds)$ for $t \in [0,T]$.
It follows from Cauchy--Schwarz inequality that $G_2(\Theta^n \circ \ubd^{-1}) = \Ebf^{\Theta^n} \int_0^T \| \dot{\ubd}(s) \|^2 \, ds \le \Ebf^{\Theta^n} \int_{\Rd \times [0,T]} \|y\|^2 \, \rhobd(dy\,ds) \le 2(M+1)$.
So $\{\Theta^n \circ \ubd^{-1}\}$ is tight and consequently $\{\Theta^n \circ (\bbd,\ubd)^{-1}\}$ is also tight.
This implies tightness of $\{\Theta^n_{(1)}\}$ since $\Theta^n_{(1)} = \Theta^n \circ (g_3(\bbd,\ubd))^{-1}$ and $g_3$ is continuous.
Finally, tightness of $\{\Theta^n_{(4)}\}$ follows since $G_3(\Theta^n_{(4)}) = R(\Theta^n_{(4)}\|\theta) \le M+1$.
This proves the tightness of $\{\Theta^n\}$.

Now suppose $\Theta^n$ converges along a subsequence (labeled as $\{n\}$ for simplicity) to $\Theta \in \Pmc(\Xi)$. 
Using arguments similar to those in the proof of Lemma~\ref{lem_kill:char_limit_1}, we now show that $\Theta \in \Pmc_\infty$ and $\Jmc(\Theta) < \infty$.
We need to show that $\Theta_{(4)} \ll \theta$ and verify properties (1)--(3) of Section~\ref{sec_kill:rate function LDP} for $\Theta$.
It is clear that (1) holds.
Using the lower semi-continuity of the function $R(\cdot\|\theta)$ and the map in \eqref{eq_kill:check2}, we have
\begin{align}
	\Jmc(\Theta) & = R(\Theta_{(4)} \|\theta) + \Ebf^\Theta \left[ \half \int_{\Rmb^d \times [0,T]} \|y\|^2 \,\rhobd(dy\,ds) \right] \notag \\
	& \le \liminf_{n \to \infty} \left( R(\Theta^n_{(4)} \|\theta) + \Ebf^{\Theta^n} \left[ \half \int_{\Rmb^d \times [0,T]} \|y\|^2 \,\rhobd(dy\,ds) \right] \right) \le M. \label{eq_kill:rate_function_pf_1}
\end{align}
This implies (2), the property $\Theta_{(4)} \ll \theta$, and that $\Jmc(\Theta) < \infty$.
Finally, recalling that the map \eqref{eq_kill:check3_1} is bounded and continuous, we have
\begin{align*}
	& \Ebf^{\Theta} \left(\left| \ti \bbd_t - \bbd_t - \int_{\Rmb^d \times [0,t]} y \,\rhobd(dy\,ds) \right| \wedge 1 \right) \\
	& \quad = \lim_{n \to \infty} \Ebf^{\Theta^n} \left(\left| \ti \bbd_t - \bbd_t - \int_{\Rmb^d \times [0,t]} y \,\rhobd(dy\,ds) \right| \wedge 1 \right) = 0
\end{align*}
for each $t \in [0,T]$, which verifies (3).
Thus we have shown that $\Theta \in \Pmc_\infty$ and $\Jmc(\Theta) < \infty$.


We will now use arguments similar to those in Lemma~\ref{lem_kill:tightness_2} to show that $\{\varpi_{\Theta^n}\}$ is pre-compact in $\Cmb([0,T]:\Mmc(\Rd))$.
	The tightness of $\{\varpi_{\Theta^n}(t)\}$ in $\Mmc(\Rd)$ for each $t \in [0,T]$
	follows from the estimate
	\begin{align}
		G_4(\varpi_{\Theta^n}(t)) 
		 \le 2\Ebf^{\Theta^n} \left[\|\bbd_t\|^2 \right] + 2T \Ebf^{\Theta^n} \left[ \int_{\Rd \times [0,T]} \| y \|^2 \, \rhobd(dy\,ds) \right] 
		\le 2dT + 4T(M+1). \label{eq_kill:rate_function_G4}
	\end{align}
	Next we proceed to consider fluctuations. 
	Fix $\delta > 0$ and consider $0 \le t_1 \le t_2 \le T$ with $|t_1 - t_2| < \delta$.
	As in the proof of
	Lemma~\ref{lem_kill:tightness_2} we can estimate
	\begin{align*}
		 &d_{BL}(\varpi_{\Theta^n}(t_1), \varpi_{\Theta^n}(t_2))\\ 
		&\quad  \le \sup_{\|f\|_{BL} \le 1} \Ebf^{\Theta^n} \left| f(\bbdtil(t_1)) - f(\bbdtil(t_2)) \right| \one_{\left\{\sigmabdtil > \int_0^{t_1} \lan\zeta,\,\varpi_{\Theta^n}(s)\ran\,ds\right\}} \\
		& \qquad + \sup_{\|f\|_{BL} \le 1} \Ebf^{\Theta^n} \left| f(\bbdtil(t_2)) \left( \one_{\left\{\sigmabdtil > \int_0^{t_1} \lan\zeta,\,\varpi_{\Theta^n}(s)\ran\,ds\right\}} - \one_{\left\{\sigmabdtil > \int_0^{t_2} \lan\zeta,\,\varpi_{\Theta^n}(s)\ran\,ds\right\}} \right) \right| \\
		& \quad \doteq \Tmcbar_1^n + \Tmcbar_2^n.
	\end{align*}
	For $\Tmcbar_1^n$, using Cauchy--Schwarz inequality as in Lemma~\ref{lem_kill:tightness_2}
	\begin{align*}
		\sup_{|t_1-t_2|<\delta} \Tmcbar_1^n & \le \sup_{|t_1-t_2|<\delta} \left( \Ebf^{\Theta^n} \left\| \bbd(t_1) - \bbd(t_2) \right\|^2 \right)^\half + \sqrt{\delta} \left( \Ebf^{\Theta^n} \int_{\Rd \times [0,T]} \|y\|^2 \, \rhobd(dy\,ds) \right)^\half \\
		& \le \sqrt{d\delta} + \sqrt{2(M+1)\delta} \to 0,
	\end{align*}		
	uniformly in $n$, as $\delta \to 0$.

	Consider now  $\Tmcbar^n_2$. It follows from \eqref{eq_kill:zeta_2} and \eqref{eq_kill:rate_function_G4} that
	\begin{align*}
		\lan\zeta,\, \varpi_{\Theta^n}(t)\ran 
		 \le \Ctil_\zeta \Ebf^{\Theta^n} (1 + \|\bbdtil_t\|^2) 
		 \le \Ctil_\zeta (1 + 2dT + 4T(M+1)) \
		 \doteq \Cbar_0. 
	\end{align*}
	So we have
	\begin{equation*}
		\left| \int_0^{t_1} \lan\zeta,\,\varpi_{\Theta^n}(s)\ran\,ds - \int_0^{t_2} \lan\zeta,\,\varpi_{\Theta^n}(s)\ran\,ds \right| \le \Cbar_0 \delta
	\end{equation*}	
	and hence
	\begin{align*}
		\sup_{|t_1-t_2|<\delta} \Tmcbar^n_2 & \le \sup_{|t_1-t_2|<\delta} \Theta^n \left( \int_0^{t_1} \lan\zeta,\,\varpi_{\Theta^n}(s)\ran\,ds < \sigmabdtil \le \int_0^{t_2} \lan\zeta,\,\varpi_{\Theta^n}(s)\ran\,ds \right) \\
		& \le \max_{j \in \Nmb_0} \Theta^n_{(4)}([j \Cbar_0 \delta, (j+2) \Cbar_0 \delta]).
	\end{align*}
Since $\Theta^n \to \Theta \in \Pmc_\infty$, it then follows from  \eqref{eq:eq1147} that
	\begin{align*}
		\limsup_{n \to \infty} \sup_{|t_1-t_2|<\delta}  \Tmcbar_2^n & \le \limsup_{n \to \infty} \max_{j \in \Nmb_0} \Theta^n_{(4)}([j \Cbar_0 \delta, (j+2) \Cbar_0 \delta]) \\
		& = \max_{j \in \Nmb_0} \Theta_{(4)}([j \Cbar_0 \delta, (j+2) \Cbar_0 \delta]) \to 0,
	\end{align*}
	as $\delta \to 0$, where the  convergence to $0$ is argued using an estimate as in \eqref{eq:eq1209}.
	Combining above two convergence results of $\Tmcbar_1^n$ and $\Tmcbar_2^n$ gives us pre-compactness of $\{\varpi_{\Theta^n}\}$ in $\Cmb([0,T]:\Mmc(\Rd))$.

	Next, let $\varpi_{\Theta^n}$ converge to $\pi^*$ along a further subsequence (labeled once again as $\{n\}$).
	Then we have $(\Theta^n,\varpi_{\Theta^n}) \to (\Theta,\pi^*)$ in $\Pmc(\Xi) \times \Cmb([0,T]:\Mmc(\Rd))$.
	Fix $f \in \Cmb_b(\Rd)$ and $t \in [0,T]$.
	Recall the function $g_t$ associated with $f$ defined in \eqref{eq_kill:char_limit_2_g}.
	Note that, for $t \in [0,T]$,
	\begin{equation*}
		\lan f,\,\varpi_{\Theta^n}(t) \ran = g_t(\Theta^n, \varpi_{\Theta^n}), \quad \lim_{n \to \infty} \lan f,\,\varpi_{\Theta^n}(t) \ran = \lan f,\,\pi^*(t) \ran.
	\end{equation*}
	We claim that $g_t(\Theta^n, \varpi_{\Theta^n}) \to g_t(\Theta,\pi^*)$ as $n \to \infty$ for all $t \in [0,T]$.
	Once the claim is verified, we will have $\lan f,\,\pi^*(t) \ran = g_t(\Theta,\pi^*)$ and consequently by Lemma~\ref{lem_kill:uniqueness} $\pi^*=\varpi_\Theta$ .
This along with  \eqref{eq_kill:rate_function_pf_1} will say that $I(\pi^*) = I(\varpi_\Theta) \le \Jmc(\Theta) \le M$.
	Hence $\pi^n = \varpi_{\Theta^n} \to \varpi_\Theta = \pi^* \in F_M$ completing the proof of the Proposition. 

To see the  claim that $g_t(\Theta^n, \varpi_{\Theta^n}) \to g_t(\Theta,\pi^*)$ as $n \to \infty$,
	note that
	\begin{equation}
		\label{eq_kill:rate_function_triangle}
		|g_t(\Theta^n,\varpi_{\Theta^n}) - g_t(\Theta,\pi^*)| \le |g_t(\Theta^n,\varpi_{\Theta^n}) - g_t(\Theta^n,\pi^*)| + |g_t(\Theta^n,\pi^*) - g_t(\Theta,\pi^*)|.
	\end{equation}
	Since $\Theta^n \to \Theta \in \Pmc_\infty$, by Remark~\ref{rmk_kill:Pmc_infty} we have
	\begin{equation}
			\label{eq_kill:rate_function_g1}
			\lim_{n \to \infty} |g_t(\Theta^n,\pi^*) - g_t(\Theta,\pi^*)| = 0.
	\end{equation}	
	For the first term on the right hand side of \eqref{eq_kill:rate_function_triangle}, we have
	\begin{align*}
		& |g_t(\Theta^n,\varpi_{\Theta^n}) - g_t(\Theta^n,\pi^*)| \\
		& \quad \le \Ebf^{\Theta^n} \left| f(\bbdtil_t) \left( \one_{\left\{\sigmabdtil > \int_0^t \lan\zeta,\,\varpi_{\Theta^n}(s)\ran \, ds \right\}} - \one_{\left\{\sigmabdtil > \int_0^t \lan\zeta,\,\pi^*(s)\ran \, ds \right\}} \right) \right| \\
		& \quad \le \|f\|_\infty \Theta^n \left( \sigmabdtil \mbox{ lies between } \int_0^t \lan\zeta,\,\varpi_{\Theta^n}(s)\ran \, ds \mbox{ and } \int_0^t \lan\zeta,\,\pi^*(s)\ran \, ds\right).
	\end{align*}
	Similar to the proof of \eqref{eq_kill:char_limit_2_claim}, with $\zeta_K$ defined as above \eqref{eq_kill:char_limit_2_claim_1}, one can show using \eqref{eq_kill:zeta_K} and \eqref{eq_kill:rate_function_G4} that
	for $t \in [0,T]$,
	\begin{align*}
		& \lim_{n \to \infty} \left| \int_0^t \lan\zeta_K,\,\varpi_{\Theta^n}(s)\ran \, ds - \int_0^t \lan\zeta_K,\,\pi^*(s)\ran \, ds \right| = 0, \\
		& \lim_{K \to \infty} \sup_{n \in \Nmb} \left| \int_0^t \lan \zeta,\, \varpi_{\Theta^n}(s) \ran\,ds - \int_0^t \lan \zeta_K,\, \varpi_{\Theta^n}(s) \ran\,ds \right| = 0, \\
		& \lim_{K \to \infty} \left| \int_0^t \lan \zeta,\, \pi^*(s) \ran\,ds - \int_0^t \lan \zeta_K,\, \pi^*(s) \ran\,ds \right| = 0.
	\end{align*}
	Hence  as $n \to \infty$,
	\begin{equation*}
		\left| \int_0^t \lan\zeta,\,\varpi_{\Theta^n}(s)\ran \, ds - \int_0^t \lan\zeta,\,\pi^*(s)\ran \, ds \right| \to 0.
	\end{equation*}
	So for any $\delta > 0$, using \eqref{eq:eq1144} and \eqref{eq:eq1147}
	\begin{align*}
		& \limsup_{n \to \infty} |g_t(\Theta^n,\varpi_{\Theta^n}) - g_t(\Theta^n,\pi^*)| \\
		& \quad \le \|f\|_\infty \limsup_{n \to \infty} \Theta^n \left( \sigmabdtil \mbox{ lies between } \int_0^t \lan\zeta,\,\varpi_{\Theta^n}(s)\ran \, ds \mbox{ and } \int_0^t \lan\zeta,\,\pi^*(s)\ran \, ds\right) \\
		& \quad \le \|f\|_\infty \limsup_{n \to \infty} \max_{j \in \Nmb_0} \Theta^n_{(4)}([j\delta,(j+1)\delta]) \\
		& \quad = \|f\|_\infty \max_{j \in \Nmb_0} \Theta_{(4)}([j\delta,(j+1)\delta]).
	\end{align*}
	Letting $\delta \to 0$ in the above display and using a similar estimate as in \eqref{eq:eq1209} we have that  $\limsup_{n \to \infty} |g_t(\Theta^n,\varpi_{\Theta^n}) - g_t(\Theta^n,\pi^*)| = 0$.
	This together with \eqref{eq_kill:rate_function_triangle} and \eqref{eq_kill:rate_function_g1} proves the claim that for every $t\in[0,T]$, $g_t(\Theta^n, \varpi_{\Theta^n}) \to g_t(\Theta,\pi^*)$ as $n \to \infty$.
\end{proof}
	
\section{Proof of Proposition \ref{prop_kill:rep_general}}
\label{sec:pfprop4.3}
Our starting point for the proof of Proposition \ref{prop_kill:rep_general} will be the representation in \eqref{eq_kill:rep_start}. We first consider the simpler setting where for fixed
$\Pi \in \Pmc(\Smb)$ and a probability space $(\Omegabar,\Fmcbar,\Pbdbar)$ as in Section \ref{sec_kill:var_rep_general}, the filtration $\hat \Fmc_t$
is replaced by $\Fmcbar_t$ where recall that
$\Fmcbar_t \doteq \sigma\{\Fmcbar_t^{\Wbar,\Xbar} \cup \Nmcbar\}$,  $\Nmcbar$ is the collection of all $\Pbdbar$-null sets, and $\Fmcbar_t^{\Wbar,\Xbar} \doteq \sigma\{\Xbar, \Wbar(s):s \le t\}$.
We write 
$\bar \Upsilon_{\Pi} \doteq (\Omegabar,\Fmcbar,\{\Fmcbar_t\},\Pbdbar)$
and consider the collection of processes
\begin{align*}
	\Amc(\bar \Upsilon_{\Pi}) \doteq \{ \psibar : \mbox{the process } \psibar & \mbox{ is } \Fmcbar_t \mbox{-progressively measurable} \\ & \qquad \qquad \mbox{ and } \Ebfbar \int_0^T \|\psibar(s)\|^2 \,ds < \infty\}.
\end{align*}
For each $N < \infty$, let
\begin{equation}
	\label{eq_kill:Amcbar^N}
	\Amc^N(\bar \Upsilon_{\Pi}) \doteq \left\{\psibar \in \Amc(\bar \Upsilon_{\Pi}) : \int_0^T \|\psibar(s)\|^2 \, ds \le N, \Pbdbar \mbox{-a.s.} \right\}.
\end{equation}
For  $\Amc \subset \Amc(\bar \Upsilon_{\Pi})$ and $g \in \Mmb_b(\Cmc \times \Smb)$, define
\begin{equation}
	\label{eq_kill:Lambda_Pi}
	\Lambda(\Amc,g) \doteq \inf_{\psibar \in \Amc} \Ebfbar \left[ \half \int_0^T \|\psibar(s)\|^2 \, ds + g\left(\Wbar + \int_0^\cdot \psibar(s) \, ds, \Xbar \right) \right].
\end{equation} 
The following lemma gives a convenient approximation result. 
The proof is  similar to that of \cite[Lemma 3.4]{BudhirajaDupuis2000variational} (see also proof of \cite[Theorem 3.1]{BoueDupuis1998variational}), but we provide details for completeness.
Recall from Section \ref{sec_kill:var_rep_general} that $\Pbdtil$ is the Wiener measure on $(\Omegatil, \Fmctil) = (\Cmc,\Bmc(\Cmc))$.
\begin{lemma}
	\label{lem_kill:continuity of repn}
	Fix $\Pi \in \Pmc(\Smb)$. Let $\{f_n\}$ be a uniformly bounded sequence of real-valued measurable functions on $\Cmc \times \Smb$ converging to $f$ a.s.\ $\Pbdtil \times \Pi$.
	Then for every $N < \infty$, $\Lambda(\Amc^N(\bar \Upsilon_{\Pi}),f_n) \to \Lambda(\Amc^N(\bar \Upsilon_{\Pi}),f)$ as $n \to \infty$.
\end{lemma}

\begin{proof}
	Fix $\eps > 0$.
	For each $n \in \Nmb$ we pick an element $\psibar_{n,\eps} \in \Amc^N(\bar \Upsilon_{\Pi})$ such that
	\begin{equation*}
		\Lambda(\Amc^N,f_n) \ge \Ebfbar \left[ \half \int_0^T \|\psibar_{n,\eps}(s)\|^2 \, ds + f_n\left(\Wbar + \int_0^\cdot \psibar_{n,\eps}(s) \, ds, \Xbar \right) \right] - \eps.
	\end{equation*}
	By definition, for each $n \in \Nmb$
	\begin{equation*}
		\Ebfbar \left[ \half \int_0^T \|\psibar_{n,\eps}(s)\|^2 \, ds + f\left(\Wbar + \int_0^\cdot \psibar_{n,\eps}(s) \, ds, \Xbar \right) \right] \ge \Lambda(\Amc^N(\bar \Upsilon_{\Pi}),f).
	\end{equation*}	
	We claim that as $n \to \infty$,
	\begin{equation*}
		\Ebfbar f_n\left(\Wbar + \int_0^\cdot \psibar_{n,\eps}(s) \, ds, \Xbar \right) - \Ebfbar f\left(\Wbar + \int_0^\cdot \psibar_{n,\eps}(s) \, ds, \Xbar \right) \to 0.
	\end{equation*}
	Suppose for now that the claim holds.
	Combining above three displays we have $\liminf_{n \to \infty} \Lambda(\Amc^N(\bar \Upsilon_{\Pi}),f_n) \ge \Lambda(\Amc^N(\bar \Upsilon_{\Pi}),f) - \eps$.
	Since $\eps > 0$ is arbitrary, this shows $\liminf_{n \to \infty} \Lambda(\Amc^N(\bar \Upsilon_{\Pi}),f_n) \ge \Lambda(\Amc^N(\bar \Upsilon_{\Pi}),f)$.
	Now in order to prove the claim, from  \cite[Lemma 2.8(b)]{BoueDupuis1998variational}, it suffices to show that the relative entropies
	\begin{equation*}
		R\left( \Pbdbar \circ \left(\Wbar + \int_0^\cdot \psibar_{n,\eps}(s) \, ds, \Xbar \right)^{-1} \Big\|\, \Pbdtil \times \Pi \right)
	\end{equation*}
	are uniformly bounded in $n$.
	For this, first consider the probability measure $\Pbdbar^{n,\eps}$ defined by
	\begin{equation*}
		\frac{d\Pbdbar^{n,\eps}}{d\Pbdbar} \doteq \exp \left\{ -\int_0^T \psibar_{n,\eps}(s) \, d\Wbar(s) - \half \int_0^T \|\psibar_{n,\eps}(s)\|^2 \, ds \right\}.
	\end{equation*}	
	By Girsanov's theorem, on the probability space $(\Omegabar,\Fmcbar,\Pbdbar^{n,\eps})$, $\Wbar + \int_0^\cdot \psibar_{n,\eps}(s) \, ds$ is an $\{\Fmcbar_t\}$-Brownian motion independent of the $\Fmcbar_0$-measurable random variable $\Xbar$, and $\Pbdbar^{n,\eps} \circ (\Xbar)^{-1} = \Pi$.
	So we have 
	\begin{equation*}
		\Pbdbar^{n,\eps} \circ \left(\Wbar + \int_0^\cdot \psibar_{n,\eps}(s) \, ds, \Xbar \right)^{-1} = \Pbdtil \times \Pi
	\end{equation*}
	and hence, for every $n \in \Nmb$,
	\begin{align*}
		& R\left( \Pbdbar \circ \left(\Wbar + \int_0^\cdot \psibar_{n,\eps}(s) \, ds, \Xbar \right)^{-1} \Big\|\, \Pbdtil \times \Pi \right) \\
		& \qquad \le R(\Pbdbar \| \Pbdbar^{n,\eps}) = \Ebfbar \left[ \half \int_0^T \|\psibar_{n,\eps}(s)\|^2 \, ds \right] \le \frac{N}{2}.
	\end{align*}
	Thus the claim holds and we have $\liminf_{n \to \infty} \Lambda(\Amc^N(\bar \Upsilon_{\Pi}),f_n) \ge \Lambda(\Amc^N(\bar \Upsilon_{\Pi}),f)$.
	
	We now prove $\limsup_{n \to \infty} \Lambda(\mathscr{A}^N(\bar \Upsilon_\Pi),f_n) \le \Lambda(\mathscr{A}^N(\bar \Upsilon_\Pi),f)$.
	Pick an element $\psibar_\eps \in \Amc^N(\bar \Upsilon_{\Pi})$ such that
	\begin{equation*}
		\Ebfbar \left[ \half \int_0^T \|\psibar_\eps(s)\|^2 \, ds + f\left(\Wbar + \int_0^\cdot \psibar_\eps(s) \, ds, \Xbar \right) \right] \le \Lambda(\Amc^N(\bar \Upsilon_{\Pi}),f) + \eps.
	\end{equation*}
	Since $\Pbdbar \circ \left(\Wbar + \int_0^\cdot \psibar_\eps(s) \, ds, \Xbar \right)^{-1} \ll \Pbdtil \times \Pi$, we have that as $n \to \infty$,
	\begin{equation*}
		\Ebfbar f_n\left(\Wbar + \int_0^\cdot \psibar_\eps(s) \, ds, \Xbar \right) \to \Ebfbar f\left(\Wbar + \int_0^\cdot \psibar_\eps(s) \, ds, \Xbar \right).
	\end{equation*}
	Also, by definition, for each $n \in \Nmb$
	\begin{equation*}
		\Lambda(\Amc^N(\bar \Upsilon_{\Pi}),f_n) \le \Ebfbar \left[ \half \int_0^T \|\psibar_\eps(s)\|^2 \, ds + f_n\left(\Wbar + \int_0^\cdot \psibar_\eps(s) \, ds, \Xbar \right) \right].
	\end{equation*}
	Combining the last three displays gives $\limsup_{n \to \infty} \Lambda(\Amc^N(\bar \Upsilon_{\Pi}),f_n) \le \Lambda(\Amc^N(\bar \Upsilon_{\Pi}),f)$ $+ \eps$.
	The result follows since $\eps > 0$ is arbitrary.
\end{proof}

\begin{remark}\label{Rk9.1}
	From the proof above, we see that Lemma~\ref{lem_kill:continuity of repn} remains valid even if $\{\Fmcbar_t\}$ is replaced by the general $\{\Fmchat_t\}$
	as long as $\Wbar$ is a Brownian motion with respect to this filtration and $\Xbar$ is measurable with respect to $\Fmchat_0$.	
\end{remark}

The following lemma proves Proposition  \ref{prop_kill:rep_general} in the special case where the general filtration  $\hat \Fmc_t$
is replaced by $\Fmcbar_t$.
\begin{lemma} 
	\label{lem_kill:rep_single}
	Let $f \in \Mmb_b(\Cmc \times \Smb)$. Then
	\begin{align} 
		& -\log \Ebf \left[\exp \left(-f(B,X)\right) \right] \notag \\
		& \quad = \inf_{\Pi, \bar \Upsilon_{\Pi}} \inf_{\psibar \in \Amc(\bar \Upsilon_{\Pi})} \left\{ R(\Pi\|\rho) + \Ebfbar \left[ \half \int_0^T \|\psibar(s)\|^2 \, ds + f\left(\Wbar + \int_0^\cdot \psibar(s) \, ds, \Xbar \right) \right] \right\}, \label{eq_kill:rep_single}
	\end{align}
	where the outer infimum is over all $\Pi \in \Pmc(\Smb)$ and all systems $\bar \Upsilon_{\Pi}$ of the form introduced at the beginning of the section.
\end{lemma}

\begin{proof}
	In view of \eqref{eq_kill:rep_complex}, it suffices to prove that for every $\Pi \in \Pmc(\Smb)$ such that $R(\Pi \| \rho) < \infty$,
	\begin{equation}
		\label{eq_kill:rep_single_key}
		\int_{\Smb} \ftil(x) \Pi(dx) = \inf_{\psibar \in \Amc(\bar \Upsilon_{\Pi})} \Ebfbar \left[ \half \int_0^T \|\psibar(s)\|^2 \, ds + f\left(\Wbar + \int_0^\cdot \psibar(s) \, ds, \Xbar \right) \right],
	\end{equation}	
	where $\ftil$ is defined in \eqref{eq_kill:ftil(xbd)_def}.
	We claim that it suffices to prove \eqref{eq_kill:rep_single_key} for $f \in \Cmb_b(\Cmc \times \Smb)$.
	Assume for now that the claim holds and
	so we assume that $f \in \Cmb_b(\Cmc \times \Smb)$.
	We first argue that LHS $\le$ RHS in \eqref{eq_kill:rep_single_key}.
	Fix $\psibar \in \Amc(\bar \Upsilon_{\Pi})$.
	Since $\psibar$ is $\Fmcbar_t$-progressively measurable, there exists (cf.~\cite[Exercise 1.5.6]{StroockVaradhan2007multidimensional}) a $(\Bmc([0,T]) \times \Bmc(\Smb) \times \Bmc(\Cmc))/\Bmc(\Rd)$-measurable map $F \colon [0,T] \times \Smb \times \Cmc \to \Rd$ such that
	\begin{equation*}
		\psibar(s,\omegabar) = F(s, \Xbar(\omegabar), \Wbar_{s \wedge \cdot}(\omegabar)), \quad \Pbdbar\mbox{-a.s. }(s,\omegabar) \in [0,T] \times \Omegabar.
	\end{equation*}
	Recall the Wiener space $(\Omegatil,\Fmctil,\Pbdtil)$ introduced in Section \ref{sec_kill:var_rep_general}. Define a collection of processes $\{\psitil_x\}_{x \in \Smb}$ on $\Omegatil$ as
	\begin{equation*}
		\psitil_x(s,\omegatil) \doteq F(s, x, \Wtil_{s \wedge \cdot}(\omegatil)), \quad (x,s,\omegatil) \in \Smb \times [0,T] \times \Omegatil.
	\end{equation*}
	Then $\psitil_x \in \Amctil$ for $\Pi$-a.e.\ $x \in \Smb$ where $\Amctil$ was defined below \eqref{eq_kill:rep_complex}.
	By independence between $\Xbar$ and $\Wbar$, 
	\begin{align*}
		& \Ebfbar \left[ \half \int_0^T \|\psibar(s)\|^2 \, ds + f\left(\Wbar + \int_0^\cdot \psibar(s) \, ds, \Xbar \right) \right] \\
		& = \int_\Smb \Ebfbar \left[ \half \int_0^T \|F(s, x, \Wbar_{s \wedge \cdot})\|^2 \, ds + f\left(\Wbar + \int_0^\cdot F(s, x, \Wbar_{s \wedge \cdot}) \, ds, x \right) \right] \Pi(dx) \\
		& = \int_\Smb \Ebftil \left[ \half \int_0^T \|\psitil_x(s)\|^2 \, ds + f\left(\Wtil + \int_0^\cdot \psitil_x(s) \, ds, x \right) \right] \Pi(dx) \\
		& \ge \int_{\Smb} \ftil(x) \Pi(dx),
	\end{align*}
	where the last inequality is from \eqref{eq_kill:ftil(xbd)}.
	Taking the infimum over all $\psibar \in \Amc(\bar \Upsilon_{\Pi})$ in the above inequality implies that LHS $\le$ RHS in \eqref{eq_kill:rep_single_key}.	
	
	We now consider the reverse inequality.
	Fix $\eps \in (0,1)$.
	For each $x \in \Smb$, let $\psitil_{x,\eps} \in \Amctil$ be a control in \eqref{eq_kill:ftil(xbd)} such that
	\begin{equation}
		\label{eq_kill:ftil_eps}
		\ftil(x) \ge \Ebftil \left[ \half \int_0^T \|\psitil_{x,\eps}(s)\|^2 \, ds + f\left(\Wtil + \int_0^\cdot \psitil_{x,\eps}(s) \, ds, x \right) \right] - \eps.
	\end{equation}
	We will now carefully select a countable sub-collection from $\{\psitil_{x,\eps}\}_{x \in \Smb}$ and use it to construct an $\Fmcbar_t$-progressively measurable process $\psibar_\eps \in \Amc(\bar \Upsilon_{\Pi})$.
	From \eqref{eq_kill:ftil_eps} we have
	\begin{equation*}
		\sup_{x \in \Smb} \Ebftil \left[ \half \int_0^T \|\psitil_{x,\eps}(s)\|^2 \, ds \right] \le 2 \|f\|_\infty + 1. 
	\end{equation*}
	Using this it is easy to check that $\{\int_0^\cdot \psitil_{x,\eps}(s)\,ds\}_{x \in \Smb}$ is tight in $\Cmc$ and thus so is $\{\Wtil^{x,\eps} \doteq \Wtil + \int_0^\cdot \psitil_{x,\eps}(s) \, ds\}_{x \in \Smb}$.
	Then there exists a compact  $K_\eps \subset \Cmc$ such that
	\begin{equation}
		\label{eq_kill:K_eps}
		\sup_{x \in \Smb} \Pbdtil\left(\Wtil^{x,\eps} \notin K_\eps\right) \le \eps/\|f\|_\infty.
	\end{equation}
	Let $\Ktil_\eps$ be a compact subset of $\Smb$ such that $$\Pi(\Smb \,\backslash \Ktil_\eps) \le \eps/\|f\|_\infty.$$
	Since $f$ is continuous, so is $\ftil$.
	In particular, $f$ is uniformly continuous on $K_\eps \times \Ktil_\eps$ and $\ftil$ is uniformly continuous on $\Ktil_\eps$.	
	So there exists some $M_\eps \in \Nmb$ and a finite partition $\{B_i : i = 1,\dotsc,M_\eps\}$ of $\Ktil_\eps$ such that
	\begin{equation}
		\label{eq_kill:uniform_cont}
		\max_{i=1,\dotsc,M_\eps} \sup_{\phi \in K_\eps, x_1,x_2 \in B_i} |f(\phi,x_1) - f(\phi,x_2)| + |\ftil(x_1) - \ftil(x_2)| \le \eps.
	\end{equation}
	Now for each $i=1,\dotsc,M_\eps$, fix a $y_i \in B_i$ and define
	\begin{equation*}
		\psibar_\eps(s,\omegabar) \doteq
		\begin{cases}
			\psitil_{y_i,\eps}(s,\Wbar(\omegabar)), & \omegabar \in \Xbar^{-1}(B_i), \\
			\hfil 0, & \omegabar \notin \Xbar^{-1}(\Ktil_\eps).
		\end{cases}
	\end{equation*}
	Since $\psitil_{y_i,\eps}$ is $\Fmctil_t$-progressively measurable for each $i = 1,\dotsc,M_\eps$ and $\Xbar$ is $\Fmcbar_0$-measurable, $\psibar_\eps$ is $\Fmcbar_t$-progressively measurable.
	For $x \in B_i$, $i \in \{1,\dotsc,M_\eps\}$, we have
	\begin{align*}
		\Ebfbar f\left(\Wbar + \int_0^\cdot \psitil_{y_i,\eps}(s,\Wbar) \, ds, x \right) & = \Ebftil f\left(\Wtil^{y_i,\eps} ,x\right) \\
		& \le \Ebftil \left[ f\left(\Wtil^{y_i,\eps}, x\right) \one_{\{\Wtil^{y_i,\eps} \in K_\eps \}} \right] + \eps \\
		& \le \Ebftil \left[ f\left(\Wtil^{y_i,\eps}, y_i\right) \one_{\{\Wtil^{y_i,\eps} \in K_\eps \}} \right] + 2\eps \\
		& \le \Ebftil f\left(\Wtil^{y_i,\eps}, y_i\right) + 3\eps,		
	\end{align*}
	where the first and third inequalities use \eqref{eq_kill:K_eps}, and the second inequality follows from \eqref{eq_kill:uniform_cont}.	
	Using this, the definition of $\Ktil_\eps$, and the independence between $\Wbar$ and $\Xbar$, we have
	\begin{align*}
		& \Ebfbar f\left(\Wbar + \int_0^\cdot \psibar_\eps(s) \, ds, \Xbar \right) \\
		& \:\:\:\:\:\: = \sum_{i=1}^{M_\eps} \Ebfbar \left[ \one_{\{\Xbar \in B_i\}} f\left(\Wbar + \int_0^\cdot \psibar_\eps(s) \, ds, \Xbar \right) \right] + \Ebfbar \left[ \one_{\{\Xbar \notin \Ktil_\eps\}} f\left(\Wbar + \int_0^\cdot \psibar_\eps(s) \, ds, \Xbar \right) \right] \\
		& \:\:\:\:\:\: \le \sum_{i=1}^{M_\eps} \int_{B_i} \Ebfbar f\left(\Wbar + \int_0^\cdot \psitil_{y_i,\eps}(s,\Wbar) \, ds, x \right) \Pi(dx) + \eps \\
		& \:\:\:\:\:\: \le \sum_{i=1}^{M_\eps} \int_{B_i} \Ebftil f\left(\Wtil^{y_i,\eps}, y_i\right) \Pi(dx) + 4\eps.
	\end{align*}
	Also note that 
	\begin{align*}
		\Ebfbar \int_0^T \|\psibar_\eps(s)\|^2 \, ds & = \sum_{i=1}^{M_\eps} \Ebfbar \left[ \one_{\{\Xbar \in B_i\}} \int_0^T \|\psitil_{y_i,\eps}(s,\Wbar)\|^2 \, ds \right] \\
		& = \sum_{i=1}^{M_\eps} \int_{B_i} \Ebftil \int_0^T \|\psitil_{y_i,\eps}(s)\|^2 \, ds \, \Pi(dx).
	\end{align*}	
	Combining above two displays, we have the following estimate on the RHS in \eqref{eq_kill:rep_single_key}
	\begin{align*}
		& \Ebfbar \left[ \half \int_0^T \|\psibar_\eps(s)\|^2 \, ds + f\left(\Wbar + \int_0^\cdot \psibar_\eps(s) \, ds, \Xbar \right) \right] \\
		& \quad \le \sum_{i=1}^{M_\eps} \int_{B_i} \Ebftil \left[ \half \int_0^T \|\psitil_{y_i,\eps}(s)\|^2 \, ds + f\left(\Wtil^{y_i,\eps}, y_i\right) \right] \Pi(dx) + 4\eps.
	\end{align*}	
	It then follows from \eqref{eq_kill:ftil_eps}, \eqref{eq_kill:uniform_cont} and definition of $\Ktil_\eps$ that the above display can be bounded above by
	\begin{equation*}
		\sum_{i=1}^{M_\eps} \int_{B_i} \ftil(y_i) \, \Pi(dx) + 5\eps \le \sum_{i=1}^{M_\eps} \int_{B_i} \ftil(x) \, \Pi(dx) + 6\eps \le \int_\Smb \ftil(x) \, \Pi(dx) + 7\eps.
	\end{equation*}
	Since $\eps > 0$ is arbitrary we have RHS $\le$ LHS in \eqref{eq_kill:rep_single_key} and this completes the proof.	
	
	It remains to prove the claim that in proving \eqref{eq_kill:rep_single_key} we can assume that $f \in \Cmb_b(\Cmc \times \Smb)$. The proof of this is similar to arguments in \cite{BoueDupuis1998variational, BudhirajaDupuis2000variational} but we provide the details.
	First note that from   \cite[Theorem V.16a]{Doob1994measure} we can find $\{f_n\} \subset \Cmb_b(\Cmc \times \Smb)$ such that $\|f_n\|_\infty \le \|f\|_\infty$ and $f_n$ converges to $f$ a.s.\ $\Pbdtil \times \Pi$.
	It then follows from dominated convergence theorem that as $n \to \infty$, $\ftil_n \to \ftil$ a.s.\ $\Pi$ and hence
	\begin{equation*}
		\int_{\Smb} \ftil_n(x) \Pi(dx) \to \int_{\Smb} \ftil(x) \Pi(dx).
	\end{equation*}
	To prove the claim, it then remains to show that $\Lambda(\Amc(\bar \Upsilon_{\Pi}),f_n) \to \Lambda(\Amc(\bar \Upsilon_{\Pi}),f)$ as $n \to \infty$, where $\Lambda$ is as in \eqref{eq_kill:Lambda_Pi}.
	Let 
	\begin{equation*}
		\Amc_f(\bar \Upsilon_{\Pi}) \doteq \left\{\psibar \in \Amc(\bar \Upsilon_{\Pi}): \Ebfbar \int_0^T \|\psibar(s)\|^2 \, ds \le 4\|f\|_\infty \right\}.
	\end{equation*}
	From the definition of $\Lambda$ in \eqref{eq_kill:Lambda_Pi} we see that for any bounded $f$, $-\|f\|_\infty \le \Lambda(\mathscr{A}(\bar \Upsilon_\Pi),f) \le \|f\|_\infty$. 
	For any $\bar{\psi}$ such that $\bar{\bold{E}} \int_0^T \|\bar{\psi}(s)\|^2 ds > 4\|f\|_\infty$, the right hand side of \eqref{eq_kill:Lambda_Pi} (with $g$ replaced by $f$) will be strictly larger than $\|f\|_\infty$, which means that 
	such $\bar{\psi}$ does not do better than a ${\psi}$ in the class $\mathscr{A}_f(\bar \Upsilon_\Pi)$.
	Thus 
	 $$\Lambda(\Amc(\bar \Upsilon_{\Pi}),f_n) = \Lambda(\Amc_f(\bar \Upsilon_{\Pi}),f_n) \mbox{ and } \Lambda(\Amc(\bar \Upsilon_{\Pi}),f) = \Lambda(\Amc_f(\bar \Upsilon_{\Pi}),f).$$
	Fix $\eps \in (0,1)$ and choose $N > 0$ such that $4\|f\|_\infty^2/N \le \eps/2$.
	Fix $\psibar \in \Amc_f(\bar \Upsilon_{\Pi})$ and define the stopping time
	\begin{equation*}
		\tau_N(\omegabar) \doteq \inf \left\{t \in [0,T] : \int_0^t \|\psibar(s,\omegabar)\|^2 \, ds \ge N \right\} \wedge T, \; \omegabar \in \Omegabar.
	\end{equation*}
	Let $\psibar_N(s) \doteq \psibar(s) \one_{[0,\tau_N]}(s)$.
	Then $\psibar_N \in \Amc^N(\bar \Upsilon_{\Pi})$, where $\Amc^N(\bar \Upsilon_{\Pi})$ is defined in \eqref{eq_kill:Amcbar^N}, and 
	\begin{equation*}
		\Pbdbar(\psibar_N \ne \psibar) \le \Pbdbar(\tau_N < T) \le \Pbdbar \left( \int_0^T \|\psibar(s)\|^2 \, ds \ge N \right) \le 4\|f\|_\infty/N.
	\end{equation*}	
	Using the above estimate, we have
	\begin{align*}
		\Lambda(\Amc^N(\bar \Upsilon_{\Pi}),f_n) & \le \Ebfbar \left[ \half \int_0^T \|\psibar_N(s)\|^2 \, ds + f_n\left(\Wbar + \int_0^\cdot \psibar_N(s) \, ds, \Xbar \right) \right] \\
		& \le \Ebfbar \left[ \half \int_0^T \|\psibar(s)\|^2 \, ds + f_n\left(\Wbar + \int_0^\cdot \psibar(s) \, ds, \Xbar \right) \right] + \eps.
	\end{align*}
	Taking the infimum over all $\psibar \in \Amc_f(\bar \Upsilon_{\Pi})$ in the above inequality we have
	\begin{equation*}
		\Lambda(\Amc_f(\bar \Upsilon_{\Pi}),f_n) = \Lambda(\Amc(\bar \Upsilon_{\Pi}),f_n) \le \Lambda(\Amc^N(\bar \Upsilon_{\Pi}),f_n) \le \Lambda(\Amc_f(\bar \Upsilon_{\Pi}),f_n) + \eps.
	\end{equation*}
	Exactly the same argument with $f_n$ replaced by $f$ gives
	\begin{equation*}
		\Lambda(\Amc_f(\bar \Upsilon_{\Pi}),f) \le \Lambda(\Amc^N(\bar \Upsilon_{\Pi}),f) \le \Lambda(\Amc_f(\bar \Upsilon_{\Pi}),f) + \eps.
	\end{equation*}	
	From Lemma~\ref{lem_kill:continuity of repn} we have that as $n \to \infty$, $$\Lambda(\Amc^N(\bar \Upsilon_{\Pi}),f_n) \to \Lambda(\Amc^N(\bar \Upsilon_{\Pi}),f).$$
	This proves the claim since $\eps > 0$ is arbitrary.
\end{proof}

Finally we now give the proof of Proposition \ref{prop_kill:rep_general} which extends the result in Lemma \ref{lem_kill:rep_single} to a setting with a general filtration. 

{\bf Proof of Proposition \ref{prop_kill:rep_general}.}
	Fix $\Pi \in \Pmc(\Smb)$ such that $R(\Pi \| \rho) < \infty$. Let
	$\Upsilon_{\Pi} \doteq (\Omegabar,\Fmcbar,\{\Fmchat_t\},\Pbdbar)$  be as in the statement of the proposition and let
	$\bar\Upsilon_{\Pi} \doteq (\Omegabar,\Fmcbar,\{\Fmcbar_t\},\Pbdbar)$.
	It suffices to prove that 
	\begin{align} 
		& \inf_{\psibar \in \Amc(\bar \Upsilon_{\Pi})} \Ebfbar \left[ \half \int_0^T \|\psibar(s)\|^2 \, ds + f\left(\Wbar + \int_0^\cdot \psibar(s) \, ds, \Xbar \right) \right] \notag \\
		& \qquad = \inf_{\psihat \in \Amc(\Upsilon_{\Pi})} \Ebfbar \left[ \half \int_0^T \|\psihat(s)\|^2 \, ds + f\left(\Wbar + \int_0^\cdot \psihat(s) \, ds, \Xbar \right) \right]. \label{eq_kill:rep_general_1}
	\end{align}
	We first claim that it suffices to prove \eqref{eq_kill:rep_general_1} for $f \in \Cmb_b(\Cmc \times \Smb)$.
	The verification of this claim follows along the same lines as the proof of the claim in the proof of Lemma~\ref{lem_kill:rep_single}, and hence we only provide a sketch here.
	With $\{f_n\} \subset \Cmb_b(\Cmc \times \Smb)$ introduced as below \eqref{eq_kill:rep_single_key}, it was shown in proof of Lemma~\ref{lem_kill:rep_single} that $\Lambda(\Amc(\bar \Upsilon_{\Pi}),f_n) \to \Lambda(\Amc(\bar \Upsilon_{\Pi}),f)$.
	Thus it suffices to show $\Lambda(\Amc(\Upsilon_{\Pi}),f_n) \to \Lambda(\Amc(\Upsilon_{\Pi}),f)$ as $n \to \infty$, where $\Lambda(\Amc(\Upsilon_{\Pi}),g)$ is defined through \eqref{eq_kill:Lambda_Pi}.
	For each $N < \infty$, define $\Amc^N(\Upsilon_{\Pi})$ as in \eqref{eq_kill:Amcbar^N} with $\Amc(\bar \Upsilon_{\Pi})$ replaced by $\Amc(\Upsilon_{\Pi})$.
	Then by the same stopping time argument below \eqref{eq_kill:rep_single_key}, it suffices to argue that for each $N < \infty$, $\Lambda(\Amc^N(\Upsilon_{\Pi}),f_n) \to \Lambda(\Amc^N(\Upsilon_{\Pi}),f)$ as $n \to \infty$. However this holds by Remark \ref{Rk9.1}.	
	This completes the proof of the claim.

	Henceforth we will assume that $f \in \Cmb_b(\Cmc \times \Smb)$.
	It is clear that LHS $\ge$ RHS in \eqref{eq_kill:rep_general_1}.	
	For the reverse inequality, we will show that
	\begin{align}
		& \inf_{\psibar \in \Amc(\bar \Upsilon_{\Pi})} \Ebfbar \left[ \half \int_0^T \|\psibar(s)\|^2 \, ds + f\left(\Wbar + \int_0^\cdot \psibar(s) \, ds, \Xbar \right) \right] \notag \\
		& \qquad \le \Ebfbar \left[ \half \int_0^T \|\psihat(s)\|^2 \, ds + f\left(\Wbar + \int_0^\cdot \psihat(s) \, ds, \Xbar \right) \right] \label{eq_kill:rep_general_2}
	\end{align}
	for each $\psihat \in \Amc(\Upsilon_{\Pi})$.
	The proof is similar to that of   \cite[Lemma 3.5]{BudhirajaDupuis2000variational} and we only give a sketch.
	We will proceed in two steps.
	
	\textit{Step $1.\ $Simple $\psihat$.} 
	For simplicity we consider the case where $$\psihat(s) = Y \one_{[t,T]}(s),$$ where $t \in [0,T]$, $Y$ is $\Fmchat_t$-measurable, and $\|Y\| \le N < \infty$ a.s.
	The proof for a general simple process is similar.
	Consider the map $\varrho: \Cmb([0,t]:\Rd) \times \Smb \times K_N \to \Rmb$, where $K_N \doteq \{ z \in \Rd : \|z\| \le N\}$, defined as
	\begin{equation*}
		\varrho(\phi,x,y) \doteq \Ebf \left[ \frac{T-t}{2} \|y\|^2 + f\left(\phi^B + \int_0^\cdot y \one_{[t,T]}(s) \, ds, x\right) \right],
	\end{equation*}
	where 
	\begin{equation*}
		\phi^B(s) \doteq
		\begin{cases}
			\phi(s), & s \in [0,t], \\
			\phi(t) + B(s-t), & s \in [t,T].
		\end{cases} 
	\end{equation*}
	Note that $\varrho$ is bounded, and that by the dominated convergence theorem it is also continuous in $(\phi,x,y)$.
	From a classical measurable selection result (see e.g.~\cite[Corollary 10.3]{ethkur}) there exists a Borel measurable function $\varrho_1 \colon \Cmb([0,t]:\Rd) \times \Smb \to K_N$ such that $$\varrho(\phi,x,\varrho_1(\phi,x)) \le \varrho(\phi,x,y)$$ for all $(\phi,x) \in \Cmb([0,t]:\Rd) \times \Smb$ and $y \in K_N$.
	With the definition $\Wbar_{[0,t]} \doteq \{\Wbar(s)\}_{0 \le s \le t} \in \Cmb([0,t]:\Rd)$ and $\Ybar \doteq \varrho_1(\Wbar_{[0,t]}, \Xbar)$, we set $$\psibar(s) \doteq \Ybar \one_{[t,T]}(s) \in \Amc(\bar \Upsilon_{\Pi}).$$
	Then
	\begin{align*}
		& \Ebfbar \left[ \half \int_0^T \|\psihat(s)\|^2 \, ds + f\left(\Wbar + \int_0^\cdot \psihat(s)\,ds, \Xbar\right) \right] \\
		& = \Ebfbar \left\{ \Ebfbar \left[ \frac{T-t}{2} \|Y\|^2 + f\left(\Wbar + \int_0^\cdot Y \one_{[t,T]}(s) \,ds, \Xbar \right) \right] \Big|\, \Fmchat_t \right\} \\
		& = \Ebfbar \varrho(\Wbar_{[0,t]},\Xbar,Y) \\
		& \ge \Ebfbar \varrho(\Wbar_{[0,t]},\Xbar,\varrho_1(\Wbar_{[0,t]}, \Xbar)) \\
		& = \Ebfbar \varrho(\Wbar_{[0,t]},\Xbar,\Ybar) \\
		& = \Ebfbar \left\{ \Ebfbar \left[ \frac{T-t}{2} \|\Ybar\|^2 + f\left(\Wbar + \int_0^\cdot \Ybar \one_{[t,T]}(s) \,ds, \Xbar \right) \right] \Big|\, \Fmchat_t \right\} \\
		& = \Ebfbar \left[ \half \int_0^T \|\psibar(s)\|^2 \, ds + f\left(\Wbar + \int_0^\cdot \psibar(s)\,ds, \Xbar\right) \right].
	\end{align*}
	So \eqref{eq_kill:rep_general_2} holds for simple $\psihat \in \Amc(\Upsilon_{\Pi})$.	
	
	\textit{Step $2.\ $General $\psihat$.}
	Next consider an arbitrary $\psihat \in \Amc(\Upsilon_{\Pi})$.
	We can assume without loss of generality that $\Ebfbar \int_0^T \|\psihat(s)\|^2 \, ds < \infty$.
	Then (cf.~\cite[Lemma 3.2.4]{KaratzasShreve1991brownian}) there exists a sequence of simple processes $\{\psihat_n\}_{n \in \Nmb} \subset \Amc(\Upsilon_{\Pi})$ such that
	\begin{equation*}
		\lim_{n \to \infty} \Ebfbar \int_0^T \|\psihat_n(s) - \psihat(s)\|^2 \, ds = 0.
	\end{equation*}
	This implies that as $n \to \infty$,
	\begin{equation*}
		\Ebfbar \sup_{0 \le t \le T} \left\| \int_0^t \psihat_n(s) \, ds - \int_0^t \psihat(s) \, ds \right\|^2 \le T \Ebfbar \int_0^T \|\psihat_n(s) - \psihat(s)\|^2 \, ds \to 0.
	\end{equation*}	
	Combining above displays and using Step $1$, we have	
	\begin{align*}
		& \inf_{\psibar \in \Amc(\bar \Upsilon_{\Pi})} \Ebfbar \left[ \half \int_0^T \|\psibar(s)\|^2 \, ds + f\left(\Wbar + \int_0^\cdot \psibar(s)\,ds, \Xbar\right) \right] \\
		& \le \Ebfbar \left[ \half \int_0^T \|\psihat_n(s)\|^2 \, ds + f\left(\Wbar + \int_0^\cdot \psihat_n(s)\,ds, \Xbar\right) \right] \\
		& \to \Ebfbar \left[ \half \int_0^T \|\psihat(s)\|^2 \, ds + f\left(\Wbar + \int_0^\cdot \psihat(s) \, ds, \Xbar \right) \right],
	\end{align*}
	as $n \to \infty$, where the convergence holds since $f \in \Cmb_b(\Cmc \times \Smb)$.
	Thus \eqref{eq_kill:rep_general_2} holds for general $\psihat \in \Amc(\Upsilon_{\Pi})$ and this completes the proof.
\hfill \qed


\appendix

\section{Proof of Theorem \ref{thm_kill:LLN}}
\label{sec_kill:append}

Let
	$\Gamma^n \doteq \frac{1}{n} \sum_{i=1}^n \delta_{(B_i, X_i)}$.
Define $Q^n$ and $\mutil^n$ as in \eqref{eq_kill:Q^n_general} and \eqref{eq_kill:mutil^n} on $\Upsilon = (\Omega, \Fmc, \{\Fmc_t\}, \Pbd)$, with $\psibd_i^n \equiv 0$, $\Pi^n = \theta^{\otimes n}$, and $\Sbd_i^n = X_i$.
It is then clear that $\Lmc(\Gamma^n, \mu^n) = \Lmc(Q^n_{(1,4)}, \mutil^n)$.
Note that the bound in \eqref{eq_kill:tightness_assumption} holds trivially in this case.
It follows from Lemmas~\ref{lem_kill:tightness_1}--\ref{lem_kill:char_limit_2} that
	$\{(Q^n,\mutil^n)\}_{n \in \Nmb}$ is tight in $\Pmc(\Xi) \times \Dmc$, and any weak limit point $(Q^*,\mutil^*)$ is almost surely such that $Q^*_{(1,4)}=\mu_0 \otimes \theta$,  and $\mutil^* = \varpi_{Q^*}$, namely
	\begin{equation*}
		\lan f,\,\mutil^*(t)\ran = \Ebf^{Q^*} \left[ f(\bbdtil_t) \one_{\left\{\sigmabdtil > \int_0^t \lan\zeta,\,\mutil^*(s)\ran\,ds\right\}}\right], \quad \forall f \in \Cmb_b(\Rd), t \in [0,T].
	\end{equation*}
Recall $\mu(t), a(t), b(t)$ defined in Section \ref{sec_kill:model}.
From \eqref{eq_kill:ODE}  $a(t) = \exp \left\{ -\int_0^t a(s)b(s) \, ds \right\}$.
So for $f \in \Cmb_b(\Rd)$ and $t \in [0,T]$,
\begin{align*}
	\Ebf^{Q^*} \left[ f(\bbdtil_t) \one_{\left\{\sigmabdtil > \int_0^t \lan\zeta,\,\mu(s)\ran\,ds\right\}}\right] & = \lan f, \mu_0(t) \ran \exp \left\{ -\int_0^t \lan\zeta,\,\mu(s)\ran\,ds \right\} \\
	& = \lan f, \mu_0(t) \ran \exp \left\{ -\int_0^t a(s)b(s) \, ds \right\} \\
	& = a(t) \lan f, \mu_0(t) \ran \\
	& = \lan f,\,\mu(t)\ran .
\end{align*}
So from Lemma~\ref{lem_kill:uniqueness}, we must have $\mu = \varpi_{Q^*} = \mutil^*$.
This implies that $\mu^n \Rightarrow \mu$ as $n \to \infty$.
The result follows. 
\qed

%

\begin{thebibliography}{10}

\bibitem{Bodineau2012large}
T.~Bodineau and M.~Lagouge.
\newblock {Large deviations of the empirical currents for a boundary-driven
  reaction diffusion model}.
\newblock {\em The Annals of Applied Probability}, 22(6):2282--2319, 2012.

\bibitem{BoueDupuis1998variational}
M.~Bou{\'e} and P.~Dupuis.
\newblock {A variational representation for certain functionals of Brownian
  motion}.
\newblock {\em The Annals of Probability}, 26(4):1641--1659, 1998.

\bibitem{BudhirajaDupuis2000variational}
A.~Budhiraja and P.~Dupuis.
\newblock {A variational representation for positive functionals of infinite
  dimensional Brownian motion}.
\newblock {\em Probability and Mathematical Statistics}, 20(1):39--61, 2000.

\bibitem{BudhirajaDupuisFischer2012}
A.~Budhiraja, P.~Dupuis, and M.~Fischer.
\newblock {Large deviation properties of weakly interacting processes via weak
  convergence methods}.
\newblock {\em The Annals of Probability}, 40(1):74--102, 2012.

	\bibitem{CGZ00}
	F.~Comets, N.~Gantert and O.~Zeitouni.
	\newblock {Quenched, annealed and functional large deviations for one-dimensional random walk in random environment}.
	\newblock {\em Probability theory and related fields}, 118(1):65--114, 2000.


\bibitem{CvMaZh}
J.~Cvitani{\'c}, J.~Ma, and J.~Zhang.
\newblock {The law of large numbers for self-exciting correlated defaults}.
\newblock {\em Stochastic Processes and their Applications}, 122(8):2781--2810,
  2012.

\bibitem{DawsonGartner1987large}
D.~A. Dawson and J.~G{\"a}rtner.
\newblock {Large deviations from the McKean-Vlasov limit for weakly interacting
  diffusions}.
\newblock {\em Stochastics: An International Journal of Probability and
  Stochastic Processes}, 20(4):247--308, 1987.

\bibitem{DemboZeitouni09}
A.~Dembo and O.~Zeitouni.
\newblock {\em { Large Deviations Techniques and Applications (2nd edition)}},
  volume~38.
\newblock Springer Science and Business Media, 2009.

\bibitem{Doob1994measure}
J.~L. Doob.
\newblock {\em {Measure Theory}}.
\newblock Graduate Texts in Mathematics 143. Springer-Verlag New York, 1994.

\bibitem{DupuisEllis2011weak}
P.~Dupuis and R.~S. Ellis.
\newblock {\em {A Weak Convergence Approach to the Theory of Large
  Deviations}}, volume 902 of {\em Wiley series in probability and mathematical
  statistics: Probability and statistics}.
\newblock John Wiley \& Sons, New York, 1997.

\bibitem{DuSpZh}
Paul Dupuis, Konstantinos Spiliopoulos, and Xiang Zhou.
\newblock Escaping from an attractor: importance sampling and rest points {I}.
\newblock {\em Ann. Appl. Probab.}, 25(5):2909--2958, 2015.

\bibitem{DuWa1}
Paul Dupuis and Hui Wang.
\newblock Importance sampling, large deviations, and differential games.
\newblock {\em Stoch. Stoch. Rep.}, 76(6):481--508, 2004.

	\bibitem{EJ15}	
	E.~Emrah and C.~Janjigian. (2015). 
	\newblock {Large deviations for some corner growth models with inhomogeneity}. 
	\newblock{\em Markov Processes and Related Fields.}, to appear and on arXiv:1509.02234.

\bibitem{ethkur}
S.N. Ethier and T.G. Kurtz.
\newblock {\em Markov Processes: Characterization and Convergence}.
\newblock Wiley, New York, 1986.

\bibitem{FengKurtz06}
J.~Feng and T.~G. Kurtz.
\newblock {\em {Large Deviations for Stochastic Processes}}.
\newblock Number 131. American Mathematical Soc, 2006.

\bibitem{Freitas99}
P.~Freitas.
\newblock { Nonlocal Reaction-diffusion Equations}.
\newblock {\em Fields Institute Comms}, 21:187--204, 1999.

\bibitem{Jona1993large}
G.~Jona-Lasinio, C.~Landim, and M.~E. Vares.
\newblock {Large deviations for a reaction diffusion model}.
\newblock {\em Probability theory and related fields}, 97(3):339--361, 1993.

\bibitem{KaratzasShreve1991brownian}
I.~Karatzas and S.~E. Shreve.
\newblock {\em {Brownian Motion and Stochastic Calculus}}, volume 113 of {\em
  Graduate Texts in Mathematics}.
\newblock Springer New York, 1991.

\bibitem{Kurtz1981approximation}
T.~G. Kurtz.
\newblock {\em {Approximation of Population Processes}}, volume~36 of {\em
  CBMS-NSF Regional Conference Series in Applied Mathematics}.
\newblock SIAM, 1981.

\bibitem{Landim2015hydrostatics}
C.~Landim and K.~Tsunoda.
\newblock {Hydrostatics and dynamical large deviations for a reaction-diffusion
  model}.
\newblock {\em arXiv preprint arXiv:1508.07818}, 2015.

\bibitem{Rachev1991probability}
S.~T. Rachev.
\newblock {\em {Probability Metrics and the Stability of Stochastic Models}}.
\newblock Wiley series in probability and mathematical statistics: Applied
  probability and statistics. John Wiley \& Sons, Chichester, 1991.

\bibitem{StroockVaradhan2007multidimensional}
D.W. Stroock and S.R.S. Varadhan.
\newblock {\em {Multidimensional Diffusion Processes}}.
\newblock Classics in Mathematics. Springer-Verlag Berlin Heidelberg, 2006.

\bibitem{Sznitman1991}
A-S. Sznitman.
\newblock {Topics in propagation of chaos}.
\newblock In {\em Ecole d'Et{\'e} de Probabilit{\'e}s de Saint-Flour
  XIX---1989}, volume 1464 of {\em Lecture Notes in Mathematics}, pages
  165--251. Springer Berlin Heidelberg, 1991.

\bibitem{Volpert14}
V.~Volpert.
\newblock {\em {Elliptic Partial Differential Equations}}.
\newblock Springer Basel, 2014.

\end{thebibliography}

\newpage
{\sc
\bigskip
\noi
A. Budhiraja\\
Department of Statistics and Operations Research\\
University of North Carolina\\
Chapel Hill, NC 27599, USA\\
email: budhiraj@email.unc.edu
\skp

\noi
W-T. Fan\\
Department of Mathematics\\
University of Wisconsin\\
Madison, WI 53705, USA\\
email: louisfan@math.wisc.edu
\skp

\noi
R. Wu\\
Division of Applied Mathematics\\
Brown University\\
Providence, RI 02912, USA\\
email: ruoyu\_wu@brown.edu

}

\end{document}